\documentclass[a4paper,11pt,reqno]{amsart}

\usepackage{mathtools}

\usepackage{amssymb}
\usepackage{mathrsfs}

\usepackage{tikz-cd}

\usepackage{adjustbox}

\usepackage[hmarginratio={1:1}]{geometry} 
\usepackage{enumitem} 

\usepackage[backend=bibtex,style=alphabetic,giveninits=true,maxnames=99]{biblatex}
\usepackage{csquotes}

\usepackage[linktocpage]{hyperref}
\usepackage{amsthm}
\usepackage[capitalize]{cleveref}

\hypersetup{colorlinks=true,linkcolor=red}

\newtheorem{theorem}{Theorem}[section]
\newtheorem{proposition}[theorem]{Proposition}
\newtheorem{corollary}[theorem]{Corollary}
\newtheorem{lemma}[theorem]{Lemma}

\newtheorem{thmx}{Theorem}

\theoremstyle{definition}
\newtheorem{definition}[theorem]{Definition}

\newtheorem{remark}[theorem]{Remark}

\crefname{theorem}{Theorem}{Theorems}
\Crefname{theorem}{Theorem}{Theorems}
\crefname{thmx}{Theorem}{Theorems}
\Crefname{thmx}{Theorem}{Theorems}
\crefname{lemma}{Lemma}{Lemmas}
\Crefname{lemma}{Lemma}{Lemmas}
\crefname{proposition}{Proposition}{Propositions}
\Crefname{proposition}{Proposition}{Propositions}
\crefname{corollary}{Corollary}{Corollaries}
\Crefname{corollary}{Corollary}{Corollaries}
\crefname{definition}{Definition}{Definitions}
\Crefname{definition}{Definition}{Definitions}
\crefname{example}{Example}{Examples}
\Crefname{example}{Example}{Examples}
\crefname{remark}{Remark}{Remarks}
\Crefname{remark}{Remark}{Remarks}
\crefname{equation}{}{}
\Crefname{equation}{}{}
\crefname{enumi}{}{}
\Crefname{enumi}{}{}

\setlist[enumerate,1]{label=(\arabic*)} 

\newlist{steps}{enumerate}{1}
\setlist[steps]{label=Step~\arabic*:,align=left,itemindent=0pt,leftmargin=1.5\parindent,labelwidth=*,labelindent=0.5\parindent}

\renewbibmacro{in:}{\ifentrytype{article}{}{\printtext{\bibstring{in}\intitlepunct}}} 

\renewbibmacro*{volume+number+eid}{\printfield{volume}\setunit*{\addcomma}\printfield{number}\setunit{\bibeidpunct}\printfield{eid}}
\DeclareFieldFormat[article]{number}{ no. #1} 

\hypersetup{breaklinks=true} 

\newcommand{\ul}[1]{\underline{#1}}
\newcommand{\ol}[1]{\overline{#1}}
\newcommand{\longto}[1]{\xrightarrow{#1}}

\newcommand{\Z}{\mathbb{Z}}
\newcommand{\kk}{\Bbbk}
\newcommand{\PP}{\mathbb{P}}

\newcommand{\compl}{\mathbb{C}}     

\newcommand{\hookto}{\hookrightarrow}

\newcommand{\verdier}{\mathbb{D}} 

\newcommand{\aff}{\mathbb{A}}    
\newcommand{\proj}{\mathbb{P}}    

\newcommand{\dcup}{\mathbin{\dot{\cup}}} 

\newcommand{\gdim}[1]{\mathrm{gl.dim}(#1)}

\DeclareMathOperator{\incl}{incl}   
\DeclareMathOperator{\id}{id}   
\DeclareMathOperator{\ev}{ev}   

\newcommand{\cat}[1]{\mathscr{#1}} 
\newcommand{\lie}[1]{\mathfrak{#1}} 

\newcommand{\Serre}{\mathbb{S}} 

\newcommand{\lmodfd}{\textrm{-}\mathbf{mod}_{\mathrm{fd}}} 

\newcommand{\Proj}{\mathbf{Proj}} 
\newcommand{\Inj}{\mathbf{Inj}} 
\newcommand{\ProjInj}{\mathbf{ProjInj}} 

\newcommand{\perv}[1]{\mathrm{Perv}\!\left(#1\right)}    

\newcommand{\Dbc}{\mathbf{D}^\mathrm{b}_\mathrm{c}} 

\newcommand{\D}{\mathbf{D}}                 
\newcommand{\Db}{\mathbf{D}^b}              

\newcommand{\K}{\mathbf{K}}                 

\newcommand{\Ch}{\mathbf{Ch}}               

\newcommand{\catO}{\mathcal{O}}             

\newcommand{\End}{\mathrm{End}}
\newcommand{\Hom}{\mathrm{Hom}}
\newcommand{\Ext}{\mathrm{Ext}}

\newcommand{\RHom}{\mathrm{RHom}}
\newcommand{\RHOM}{\mathrm{R}\mathcal{H}\mathrm{om}}

\newcommand{\Fun}{\mathbf{Fun}} 
\newcommand{\dgFun}{\mathbf{Fun}_{\mathrm{dg}}} 
\newcommand{\dgVect}{\mathbf{dgVect}}           

\newcommand{\ic}[1]{\mathrm{IC}_{#1}}

\newcommand{\op}{\mathrm{op}} 

\newcommand{\stan}[1]{\Delta_{#1}}      
\newcommand{\costan}[1]{\nabla_{#1}}    

\newcommand{\zz}[3]{\mathrm{M}^{#1}_{#2,#3}} 

\newcommand{\GL}{\mathrm{GL}}   

\newcommand{\pt}{\mathbf{pt}}   

\newcommand{\sptw}[1]{\mathsf{ST}_{#1}} 

\newcommand{\dgptw}[1]{\widetilde{\mathsf{PT}}_{#1}} 
\newcommand{\ptw}[1]{\mathsf{PT}_{#1}} 
\newcommand{\pctw}[1]{\mathsf{PT}'_{#1}} 

\DeclareMathOperator{\Sh}{Sh}       
\DeclareMathOperator{\Tw}{Tw}       

\newcommand{\cone}[1]{\mathrm{cone}(#1)}

\title[Serre functor and $\proj$-objects for perverse sheaves on $\proj^n$]{Serre functor and $\proj$-objects\\for perverse sheaves on $\proj^n$}

\author{Lukas Bonfert}
\address{L.B.: Max-Planck-Institut f\"ur Mathematik, Vivatsgasse 7, 53111 Bonn, Germany}
\email{bonfert@mpim-bonn.mpg.de}

\author{Alessio Cipriani}
\address{A.C.: Dipartimento di Informatica - Settore di Matematica, Universit\`a degli Studi di Verona, Strada le Grazie 15 - Ca’ Vignal, I-37134 Verona, Italy}
\email{alessio.cipriani@univr.it}


\begin{document}
    \begin{abstract}
        We show that the inverse Serre functor for the constructible derived category $\Dbc(\proj^n)$ is given by the $\proj$-twist at the simple perverse sheaf corresponding to the open stratum.
        Moreover, we show that all indecomposable perverse sheaves on $\proj^n$ are $\proj$-like objects, and explicitly construct morphisms spanning their total endomorphism spaces.
    \end{abstract}
    \maketitle
    \section{Introduction}
        A Serre functor on a $\kk$-linear triangulated category is an autoequivalence $\Serre\colon\cat{D}\to\cat{D}$ such that for any pair of objects $E,F\in\cat{D}$ there exists a functorial isomorphism $\Hom_{\cat{D}}(E,F)\cong \Hom_{\cat{D}}(F,SE)^{\vee}$.
        Serre functors generalize Serre duality from algebraic geometry, and are an important tool in the theory of triangulated categories.
        For instance, they can be used to construct left (resp.~right) adjoints to functors having a right (resp.~left) adjoint.
        
        Another important class of automorphisms of triangulated categories in algebraic geometry and representation theory are the spherical twists associated to spherical objects \cite{MR1831820}.
        For example, these can be used to construct braid group actions on triangulated categories, and certain functors from representation theory such as shuffling functors can be realized as spherical twists \cite{MR4239691}.
        By definition, an object $E\in\cat{D}$ is $d$-spherical if there is an isomorphism of graded algebras $\End_\cat{D}^*(E)\cong\kk[t]/(t^2)$ with $\deg(t)=d$ and $E$ is $d$-Calabi--Yau.
        The value of the spherical twist $\sptw{E}$ at $X\in\cat{D}$ is then defined by the triangle
        \begin{equation*}
            \Hom_\cat{D}^*(E,X)\otimes E\longto{\ev} X\to\sptw{E}(X)\to\Hom_\cat{D}^*(E,X)\otimes E[1].
        \end{equation*}
        
        Consider the constructible derived category $\Dbc(\proj^n)$ of the complex projective space $\proj^n$ with the usual Bruhat stratification, whose strata (the Bruhat cells) have complex dimension $0, 1,\dots, n$.
        By definition, $\Dbc(\proj^n)$ consists of those complexes of sheaves of $\kk$-vector spaces on $\proj^n$ whose cohomology is locally constant on all strata of $\proj^n$.
        By gluing the standard $t$-structures on the constructible derived category of each stratum (shifted by the dimension of the stratum) one obtains the (middle-)perverse $t$-structure on $\Dbc(\proj^n)$, and its heart is the category $\perv{\proj^n}$ of (middle-)perverse sheaves \cite{BBD}.
        This perverse $t$-structure plays an important role in representation theory since there is an equivalence $\perv{\proj^n}\cong\catO^\lie{p}_0(\lie{sl}_{n+1}(\kk))$, where $\lie{p}\subseteq\lie{sl}_{n+1}(\kk)$ denotes the parabolic Lie subalgebra with block sizes $(n,1)$. 
        As this perverse $t$-structure has faithful heart, this yields an equivalence $\Dbc(\proj^n)\cong\Dbc(\catO^\lie{p}_0(\lie{sl}_{n+1}(\kk)))$ \cite{MR1322847}.

        The category of perverse sheaves on $\proj^n$ is moreover equivalent to the category of finite-dimensional modules over an explicit finite-dimensional algebra $A_n$ \cite{MR1862802}.
        Since $A_n$ has finite global dimension, it follows from results of Happel \cite{Happel} and Bondal--Kapranov \cite{BondalKapranovSerre} that $\Dbc(\proj^n)$ admits a Serre functor, namely the left derived functor of the Nakayama functor $A_n^\vee\otimes_{A_n}-$.
        However, these results do not provide a description of the Serre functor that is intrinsic to the constructible derived category.

        In the case of the complex projective line $\proj^1=\pt\dcup\aff^1$ stratified by a point and its complement, such an intrinsic description is provided in \cite{MR2739061}.
        In this example, the category $\perv{\proj^1}$ has two simple objects corresponding to the two strata.
        Explicitly, these are the skyscraper sheaf $\ic{0}=\incl_*\ul{\kk}_{\pt}$ and the shifted constant sheaf $\ic{1}=\ul{\kk}_{\proj^1}[1]$.
        The simple perverse sheaf $\ic{1}$ is a $2$-spherical object in $\Dbc(\proj^1)$, and the inverse Serre functor for $\Dbc(\proj^n)$ is then given by $\Serre^{-1}=\sptw{\ic{1}}^2$.
        
        Our main result is a generalization of this description of the Serre functor to $\proj^n$.
        In this case, the simple perverse sheaf $\ic{n}=\ul{\kk}_{\proj^n}[n]$ corresponding to the open stratum is not a spherical object, but rather a $\proj^n$-object in the sense of Huybrechts and Thomas \cite{MR2200048}.
        Explicitly this means that there is an isomorphism of graded algebras $\End_{\proj^n}(\ic{n})\cong\kk[t]/(t^{n+1})$ with $\deg(t)=2$, and that $\ic{n}$ is $2n$-Calabi--Yau.
        The corresponding generalization of spherical twists is provided by the $\proj$-twists from \cite{MR2200048}, which are defined as certain ``double cones'', see \cref{Ptwistdef} below.
        In our situation, we obtain:
        \begin{thmx}[\cref{serrefunctortwist}]\label{serrefunctortwistintro}
            The inverse Serre functor $\Serre^{-1}$ of $\Dbc(\proj^n)$ is isomorphic to the $\proj$-twist $\ptw{\ic{n}}$.
        \end{thmx}
        This in particular recovers the result from \cite{MR2739061} for $\proj^1$, since by definition a $\proj^1$-object is the same as a $2$-spherical object, and for any $\proj^1$-object $E$ we have $\ptw{E}\cong\sptw{E}^2$.
        
        The proof of \cref{serrefunctortwistintro} relies on a characterization of the Serre functor adapted from \cite{MR2369489}, see \cref{serrefunctorcharacterization}.
        The main idea is to compare the ``candidate inverse Serre functor'' to the inverse Serre functor by studying its action on the injective and projective-injective objects.
        In \cite{MR2369489} the dual version of this criterion was used to describe the Serre functor of $\Db(\catO_0(\lie{g}))$ for any finite-dimensional complex semisimple Lie algebra $\lie{g}$.
        By an entirely formal argument this description also descends to $\Db(\catO^\lie{p}_0(\lie{g}))$ for any parabolic subalgebra $\lie{p}\subseteq\lie{g}$, and thus also to $\Dbc(\proj^n)$.
        In \cref{otherserres}, we summarize these results and their relation to the description of the Serre functor in terms of finite-dimensional algebras, and also relate our description of the Serre functor of $\Dbc(\proj^n)$ to the description of the Serre functor for the full flag variety obtained in \cite{MR2119139}.

        Motivated by \cref{serrefunctortwistintro}, one may ask whether there are further $\proj$-objects in $\perv{\proj^n}$.
        However, it is easy to see that no indecomposable object except for $\ic{n}$ and the projective-injectives can be Calabi--Yau, see \cref{CY_objects}.
        Hence any other indecomposable object $E$ can at best be $\proj^k$-like in the sense that $\End_{\proj^n}^*(E)\cong\kk[t]/(t^{k+1})$.
        
        To describe the indecomposable perverse sheaves, in \cref{strings} we inductively construct certain string objects $\zz{\pm}{a}{b}\in\perv{\proj^n}$ for $0\leq b\leq a\leq n$, starting from the simple objects and (co)standard objects.
        Alternative constructions of these objects can be found in \cite{CiprianiLanini}, where they are used to describe the wall-and-chamber structure of $\perv{\proj^n}$.
        Since $\perv{\proj^n}\cong A_n\lmodfd$ for a special biserial algebra $A_n$, the classification of indecomposable modules over special biserial algebras from \cite{ButlerRingel,WaldWaschbuesch} shows that the string objects together with the indecomposable projective-injective objects are all the indecomposable perverse sheaves.
        
        As the indecomposable projective-injective objects are $0$-spherical, our second result then shows that all indecomposable perverse sheaves are either $\proj$-like or $0$-spherical:
        \begin{thmx}[\cref{stringsplike}]\label{stringsplikeintro}
            Let $0\leq b\leq a\leq n$.
            \begin{enumerate}
                \item If $a-b$ is even, then the string objects $\zz{\pm}{a}{b}$ are $\proj^{(a+b)/2}$-like.
                \item If $a-b$ is odd, then the string objects $\zz{\pm}{a}{b}$ are $\proj^{(a-b-1)/2}$-like.
            \end{enumerate}
        \end{thmx}
        As easy consequences of \cref{stringsplikeintro}, one also obtains a classification of the spherical, spherelike and exceptional objects in $\perv{\proj^n}$, see \cref{sphericalexceptionals}.
        In particular, this recovers the classification of the exceptional objects from \cite{MR4090927}.

        The proof of \cref{stringsplikeintro} is rather technical and occupies most of \cref{SerreFunctorDescr,PObjectClassification}.
        The first step is to compute $\End_{\proj^n}^*(\zz{\pm}{a}{b})$ by chasing the long exact sequences obtained from the inductive construction of the string objects, see \cref{prel_results,stdzigzaghoms,seczigzaghoms}.
        As the base cases of this construction are the simple objects and the (co)standard objects, this requires us to explicitly fix morphisms between these objects and to determine their compositions, see \cref{simplemorph,stdsimplemorph,subsec_projective_injectives}.
        The computation of $\Hom_{\proj^n}(\zz{\pm}{a}{b},\zz{\pm}{a}{b}[2i])$ also yields canonical non-zero morphisms $\Phi^{2i}_{a,b}\colon \zz{\pm}{a}{b}\to\zz{\pm}{a}{b}[2i]$, and the final step is then to check that $\Phi^{2i}_{a,b}\Phi^2_{a,b}=\Phi^{2i+2}_{a,b}$ up to a non-zero scalar (whenever this is possible by degree reasons), see \cref{zigzagcomposition}.
        \subsection{Notation}
            Throughout the paper, $\kk$ denotes an algebraically closed field of characteristic~$0$.
            The $\kk$-linear duality functor is denoted by $(-)^\vee=\Hom_\kk(-,\kk)$.

            For a $\kk$-linear triangulated category $\cat{D}$ and $A,B\in\cat{D}$ we denote the total Hom space by $\Hom_\cat{D}^*(A,B)=\bigoplus_{r\in\mathbb Z}\Hom_\cat{D}(A,B[r])$, with the degree~$r$ part given by $\Hom_\cat{D}(A,B[r])$.
            We do not write shifts of morphisms, i.e.~we just write $f\colon A[1]\to B[1]$ for $f\colon A\to B$.

            We write $\mathrm{R}F\colon \D^+(\cat{A})\to\D^+(\cat{B})$ for the right derived functor of a left exact functor $F\colon \cat{A}\to\cat{B}$ of abelian categories.
            As usual, we will however suppress the notation for derived functors for functors arising from geometry, such as the pushforward.
            The right derived functor in the $\infty$-categorical sense will be denoted by $\mathbb{R}F\colon\D^+_\infty(\cat{A})\to\D^+_\infty(\cat{B})$.

            The following table is a list of notation for the morphisms in $\Dbc(\proj^n)$ between the simple perverse sheaves $\ic{k}$, standard objects $\stan{k}$ (see \cref{SimpleStdProj}) and string objects $\zz{+}{a}{b}\in\perv{\proj^n}$ (see \cref{strings}) which will be used throughout the paper.
            \begin{center}
                \begin{tabular}{l|l||l|l}
                    Morphism&Definition&Morphism&Definition\\
                    \hline\hline
                    $\epsilon^r_{k,l}\colon \ic{k}\to\ic{l}[r]$&\cref{simplemorph} &  $\psi_{a-2,b}\colon \zz{+}{a-2}{b}\to\stan{a}[1]$&\cref{strings}\\
                    $\mu_{k,l}\colon\stan{k}\to\ic{l}[l-k]$&\cref{topIC_to_lower_nabla} & $m^r_{a,b}\colon\zz{+}{a}{b}\to\ic{b}[r]$&\cref{HomZZplusIC}\\
                    $\delta^r_{k,l}\colon\stan{k}\to\stan{l}[r]$&\cref{DeltaHoms} & $n^r_{a,b}\colon \zz{+}{a}{b}\to\ic{a}[r]$&\cref{HomZZtopICdiag}\\
                    $\phi^r_{k,l}\colon \ic{l}\to\stan{k}[r]$&\cref{mor_simple_stand} & $\Phi^{2i}_{a,b}\colon \zz{+}{a}{b}\to\zz{+}{a}{b}[2i]$&\cref{zigzaghomdiags}\\
                    $\iota_{a,b}\colon \stan{a}\to\zz{+}{a}{b}$&\cref{strings} & $\ol{\Phi}^{2i}_{a,b}\colon \zz{+}{a-2i}{b}\to\zz{+}{a}{b}[2i]$&\cref{zigzaghomdiags}\\
                    $\pi_{a,b}\colon \zz{+}{a}{b}\to\zz{+}{a-2}{b}$&\cref{strings} & $\zeta^{2i}_{a-2i,b}\colon \stan{a-2i}\to\zz{+}{a}{b}[2i]$&\cref{zigzaghomdiags}
                \end{tabular}
            \end{center}
        \subsection*{Acknowledgements}
        We thank Martina Lanini, Catharina Stroppel and Jon Woolf for many helpful discussions about this project.
        We also thank Andreas Hochenegger and Andreas Krug for helpful explanations about Fourier--Mukai transforms.
        For computations of Ext spaces we often used Haruhisa Enomoto's FD Applet \cite{FDApplet}.
        Most of the research was done during visits of L.B. to the University of Verona and of A.C. to the University of Bonn, and we thank both institutions for their hospitality.
        
        L.B. was funded by the Max Planck Institute for Mathematics (IMPRS Moduli Spaces) and the Hausdorff Center for Mathematics, which is funded by the Deutsche Forschungsgemeinschaft (DFG, German Research Foundation) under Germany’s Excellence Strategy -- EXC-2047/1 -- 390685813.
        A.C. was funded by MUR PNRR–Seal of Excellence, CUP B37G22000800006 and supported by the “National Group for Algebraic and Geometric Structures, and their Applications” (GNSAGA - INdAM).
        We also acknowledge the project funded by NextGenerationEU under NRRP, Call PRIN 2022 No. 104 of February 2, 2022 of Italian Ministry of University and Research; Project 2022S97PMY Structures for Quivers, Algebras and Representations (SQUARE).
    \section{Setting and Background}\label{setup}
        \subsection{\texorpdfstring{$\proj$-objects and $\proj$-twists}{P-objects and P-twists}}
            We begin by recalling the $\proj$-twists at $\proj$-objects from \cite{MR2200048}, using the nomenclature from \cite{HocheneggerKrugFormality}.
            Let $\cat{D}$ be a $\kk$-linear triangulated category.
            \begin{definition}
                Let $E\in\cat{D}$ and $k\in\mathbb Z$ with $k\geq 0$.
                \begin{enumerate}
                    \item $E$ is \emph{$\PP^k$-like} if there is an isomorphism of graded $\kk$-algebras $\End_\cat{D}^*(E)\cong \kk[t]/(t^{k+1})$ with $\deg(t)=2$.
                    \item $E$ is a \emph{$\proj^k$-object} if it is $\PP^k$-like, $\Hom_\cat{D}^*(E,X)$ is finite-dimensional for all $X\in\cat{D}$, and $E$ is $2k$-Calabi--Yau.
                    \item A \emph{$\proj$-(like) object} is a $\proj^k$-(like) object for any $k$.
                \end{enumerate}
            \end{definition}
            Recall that $E\in\cat{D}$ is $d$-Calabi--Yau if there is a natural isomorphism
            \begin{equation*}
                \Hom_\cat{D}(E,-)\cong\Hom_\cat{D}(-,E[d])^\vee.
            \end{equation*}
            It is immediate from the definitions that if a $\proj^k$-like object $E$ is $d$-Calabi--Yau, then necessarily $d=2k$, cf.~\cite[Def.~2.1]{HocheneggerKrugRelations} and \cite[Rem.~1.2]{MR2200048}.
            
            Slightly more generally, \cite{Krug} and \cite{HocheneggerKrugFormality} also introduced $\proj^k[d]$-(like) objects, for which (1) is replaced by $\End_\cat{D}^*(E)\cong\kk[t]/(t^{k+1})$ with $\deg(t)=d$.
            Thus $\proj^k[2]$(-like) objects are the same as $\proj^k$(-like) objects.
            As well-known special cases, $\proj^1[d]$-objects and $\proj^1[d]$-like objects are the same as $d$-spherical objects and $d$-spherelike objects, respectively, and $\proj^0$-like objects are the same as exceptional objects.
            
            The following lemma provides a useful criterion for when $\proj^k$-like objects are $\proj^k$-objects.
            \begin{lemma}\label{check_CY_on_proj}
                Let $E\in\cat{D}$ be a $\proj^k$-like object.
                \begin{enumerate}
                    \item $E$ is $2k$-Calabi--Yau if and only if the composition pairing
                        \begin{equation*}
                            \Hom_\cat{D}(E,X)\otimes\Hom_\cat{D}(X,E[2k])\to\Hom_\cat{D}(E,E[2k])\cong\kk
                        \end{equation*}
                        is non-degenerate for all $X\in\cat{D}$.
                    \item If $\cat{D}=\Db(\cat{A})$ for an abelian category $\cat{A}$ of finite global dimension with enough projectives, then $E$ is $2k$-Calabi--Yau if and only if the composition pairing is non-degenerate for all $X=P[r]$ with $P\in\Proj(\cat{A})$ and $r\in\Z$.
                \end{enumerate}
            \end{lemma}
            \begin{proof}\leavevmode
                \begin{enumerate}
                    \item For spherelike objects, this is \cite[Lemma~2.15]{MR1831820}.
                        The same argument works for $\proj$-like objects, as it only requires $\Hom_\cat{D}(E,E[2k])\cong\kk$
                    \item For spherelike objects, this is shown in \cite[Lemma~3.3]{MR4239691}.
                        The same argument works for $\proj$-like objects, as it only requires $\Hom_\cat{D}(E,E[2k])\cong\kk$ and the first part of the lemma.\qedhere
                \end{enumerate}
            \end{proof}
            For the rest of this subsection we fix a \emph{dg enhancement} (originally called \emph{enhancement} in \cite{BondalKapranov}) of $\cat{D}$.
            By definition, a dg enhancement consists of a $\kk$-linear pretriangulated dg category $\widetilde{\cat{D}}$ (as defined in \cite[\S 3, Def.~1]{BondalKapranov}) and an equivalence of triangulated categories $H^0(\widetilde{\cat{D}})\cong\cat{D}$.
            Note that if $\widetilde{\cat{D}}$ is pretriangulated, then $\dgFun(\widetilde{\cat{D}},\widetilde{\cat{D}})$ is again a pretriangulated dg category by \cite[\S 3, Examples, 4.]{BondalKapranov}.
            In particular the cone of a morphism of dg functors is again a dg functor.
            
            If $\widetilde{\cat{D}}$ has all coproducts, then for an object $E\in\widetilde{\cat{D}}$ there is the dg functor $-\otimes E\colon \dgVect_\kk\to \widetilde{\cat{D}}$, which is defined as the right adjoint to $\Hom_{\widetilde{\cat{D}}}(E,-)\colon\widetilde{\cat{D}}\to\dgVect_\kk$.
            We denote the counit of this adjunction by $\ev\colon \Hom_{\widetilde{\cat{D}}}(E,-)\otimes E\to\id_{\widetilde{\cat{D}}}$.
            
            A generalization of spherical twists at spherical objects is provided by the $\proj$-twists at $\proj$-objects from \cite[\S 2]{MR2200048}.
            These are defined as follows:
            \begin{definition}\label{Ptwistdef}
                Let $E\in\cat{D}$ be a $\proj$-like object.
                Assume that the tensor product $\Hom_{\widetilde{\cat{D}}}(E,X)\otimes E\in\widetilde{\cat{D}}$ exists for all $X\in\widetilde{\cat{D}}$.
                \begin{enumerate}
                    \item Pick a closed morphism $\tilde{t}\in\Hom_{\widetilde{\cat{D}}}(E,E[2])$ of degree~$0$ representing an algebra generator of $\End_\cat{D}^*(E)$.
                       Define the dg functor $\dgptw{E}=\cone{\ol{\ev}}\colon\widetilde{\cat{D}}\to \widetilde{\cat{D}}$ by the following commutative diagram in $\dgFun(\widetilde{\cat{D}},\widetilde{\cat{D}})$:
                        \begin{equation}
                            \label{ptwistdiagram}
                            \begin{tikzcd}[column sep=large]
                                \bigl(\Hom_{\widetilde{\cat{D}}}(E,-)\otimes E\bigr)[-2]\arrow{r}{\tilde{t}^*\otimes\id-\id\otimes \tilde{t}}
                                &\Hom_{\widetilde{\cat{D}}}(E,-)\otimes E\arrow{d}{\ev}\arrow{r}
                                &\cone{\tilde{t}^*\otimes\id-\id\otimes \tilde{t}} \arrow[dashed]{ld}{\exists\ol{\ev}}\\
                                &\id_{\widetilde{\cat{D}}}\arrow{ld}&\\
                                \cone{\ol{\ev}}&& 
                            \end{tikzcd}
                        \end{equation}
                        Here $\ol{\ev}\colon \cone{\tilde{t}^*\otimes\id-\id\otimes \tilde{t}}\to\id_{\widetilde{\cat{D}}}$ is the canonical morphism of dg functors induced by $\ev$, which exists since $\ev\circ (\tilde{t}^*\otimes\id-\id\otimes\tilde{t})=0$.
                    \item The \emph{$\proj$-twist at $E$} is the induced triangulated functor $\ptw{E}=H^0(\dgptw{E})\colon \cat{D}\to\cat{D}$.
                \end{enumerate}
            \end{definition}
            \begin{remark}\label{remarksontwists}\leavevmode
                \begin{enumerate}
                    \item A priori, the functor $\ptw{E}\colon\cat{D}\to\cat{D}$ depends on the choices made in the definition.
                        However, as $E$ is a $\proj$-like object, the generator $t\colon E\to E[2]$ of $\End_\cat{D}^*(E)$ is unique up to non-zero scalar, and rescaling $t$ obviously results in naturally isomorphic triangles.
                        Similarly, choosing a different representative for $t$ results in quasi-isomorphic cones.
                        By \cite[Thm.~3.2]{AnnoLogvinenkoUniqueness}, the functor $\ptw{E}\colon\cat{D}\to\cat{D}$ is furthermore independent of the choice of cones and factorization $\ol{\ev}$.
                        Thus $\ptw{E}\colon\cat{D}\to\cat{D}$ is well-defined up to natural isomorphism.
                    \item The assumption that $\Hom_{\widetilde{\cat{D}}}(E,X)\otimes E$ exists for all $X\in\widetilde{\cat{D}}$ is automatically satisfied if $\widetilde{\cat{D}}$ has all coproducts, or if $\Hom_{\widetilde{\cat{D}}}(E,X)$ is finite-dimensional.
                    \item If $\cat{D}=\D^+(\cat{A})$ is the derived category of an abelian category with enough injectives, a dg enhancement of $\cat{D}$ is given by the pretriangulated dg category $\widetilde{\cat{D}}=\Ch^+(\Inj(\cat{A}))$ together with the canonical equivalence $H^0(\widetilde{\cat{D}})=\K^+(\Inj(\cat{A}))\to\D^+(\cat{A})$.
                        Hence by definition, to compute $\Hom_{\widetilde{\cat{D}}}(E,X)$ for a $\proj$-like object $E\in\cat{A}$ and any $X\in\widetilde{\cat{D}}$, one first has to replace $E$ by a (fixed) injective resolution.
                        However, the derived Hom functor $\RHom_\cat{A}(E,-)=\Hom_{\Ch^+(\cat{A})}(E,-)\colon\widetilde{\cat{D}}\to\dgVect_\kk$ is quasi-isomorphic to $\Hom_{\widetilde{\cat{D}}}(E,-)$.
                        Thus, in this situation we can use $\RHom_\cat{A}(E,-)$ instead of $\Hom_{\widetilde{\cat{D}}}(E,-)$ to define the $\proj$-twist $\ptw{E}\colon \cat{D}\to\cat{D}$.
                        This is easier to compute in practice, since here $E$ does not need to be replaced by an injective resolution.
                \end{enumerate}
            \end{remark}
            \begin{remark}\leavevmode
                \begin{enumerate}
                    \item Spherical twists at spherical objects can be generalized to spherical twists at spherical functors.
                        Similarly, $\proj$-twists at $\proj$-objects can be generalized further to $\proj$-twists at (split) $\proj$-functors, see \cite{Addington,Cautis,AnnoLogvinenkoFunctors}.
                        However, we will not use these constructions.
                    \item By \cite{MR3805198}, any autoequivalence of a triangulated category can be realized as a spherical twist at a spherical functor.
                        For $\proj$-twists at $\proj$-objects this can be carried out explicitly, see \cite[\S 4]{MR3805198}.
                \end{enumerate}
            \end{remark}
            The following main properties of $\proj$-twists were proved in \cite{MR2200048} using Fourier--Mukai transforms.
            We briefly sketch how the required properties can be shown purely in terms of dg-enhanced triangulated categories.
            That all statements carry over to the dg setup is presumably well-known to experts, see for instance \cite[Prop.~2.5]{HocheneggerKrugFormality}.
            \begin{proposition}[Huybrechts--Thomas]
                Let $E\in\cat{D}$ be a $\proj^k$-object.
                \begin{enumerate}
                    \item $\ptw{E}\colon\cat{D}\to\cat{D}$ is an equivalence.
                    \item $\ptw{E}(E)\cong E[-2k]$.
                    \item If $E$ is spherical (i.e.~if $k=1$), then $\ptw{E}\cong\sptw{E}^2$.
                \end{enumerate}
            \end{proposition}
            \begin{proof}\leavevmode
                \begin{enumerate}
                    \item A computation similar to the proof of \cite[Lemma~2.8]{MR1831820} shows that $\ptw{E}$ has a left adjoint $\pctw{E}$, which is defined dually.
                        Moreover, by a similar argument and the Calabi--Yau property it follows that $\pctw{E}$ is also right adjoint to $\ptw{E}$.
                        The claim then follows by the arguments from \cite[Prop.~2.6]{MR2200048}.
                    \item This is straightforward, see \cite[Lemma 2.5]{MR2200048} for details.
                    \item See \cite[Prop.~2.9]{MR2200048}.
                        Alternatively, this can be seen by comparing the diagram defining $\sptw{E}^2$ to the octahedral axiom diagram for the factorization in the definition of the $\proj$-twist.\qedhere
                \end{enumerate}
            \end{proof}
        \subsection{The Serre functor}
            We recall the notion of Serre functor from \cite{BondalKapranovSerre}.
            \begin{definition}
                A \emph{Serre functor} for a $\kk$-linear triangulated category $\cat{D}$ is a functor $\Serre\colon \cat{D}\to \cat{D}$ such that there are natural isomorphisms
                \begin{equation*}
                    \Hom_\cat{D}(X,Y)\cong \Hom_\cat{D}(Y,\Serre(X))^\vee.
                \end{equation*}
            \end{definition}
            As an immediate consequence of the Yoneda lemma, there is at most one Serre functor for $\cat{D}$ up to natural isomorphism \cite[Prop.~3.4]{BondalKapranovSerre}.

            For $\cat{D}=\Db(A\lmodfd)$ for a finite-dimensional algebra $A$, the Serre functor has the following description:
            \begin{proposition}[Happel]
                \label{fdalgserrefunctor}
                Let $A$ be a finite-dimensional algebra.
                Then $\Db(A\lmodfd)$ has a Serre functor if and only if $\gdim{A}<\infty$.
                In this case, the Serre functor is the left derived functor of the Nakayama functor $A^\vee\otimes_A-$.
            \end{proposition}
            \begin{proof}
                If $\gdim{A}<\infty$, then the left derived Nakayama functor is a Serre functor by \cite[Prop.~I.4.10]{Happel}.
                Conversely, it is clear that if $\Serre$ is a Serre functor, then
                \begin{equation*}
                    \Ext_A^r(X,L)\cong\Hom_{\Db(A\lmodfd)}(X,L[r])\cong\Hom_{\Db(A\lmodfd)}(L,\Serre(X)[-r])^\vee
                \end{equation*}
                for any $X\in A\lmodfd$ and any simple $A$-module $L$, and the right-hand side vanishes for $r$ large enough such that $\Serre(X)[-r]\in\cat{D}^{>0}$, where $\cat{D}^{>0}$ denotes the positive part of the standard $t$-structure.
            \end{proof}
        \subsection{The constructible derived category \texorpdfstring{$\Dbc(\PP^n)$}{Dbc(Pn)}}\label{setupDbcPn}
            We recall the constructible derived category of the complex projective space $\proj^n=\proj^n_\compl$, which will be the focus of the rest of the paper.
            The same construction and all of the tools work in much greater generality, see \cite{BBD}, \cite{HottaTakeuchiTanisaki} or \cite{Achar}.
            
            Consider $\proj^n$ with the usual stratification by Bruhat cells, i.e.~by the subspaces
            \begin{equation*}
                S_k=\{[x_0:\dots:x_{k-1}:1:0:\dots :0]\}\subseteq\proj^n
            \end{equation*}
            for $0\leq k\leq n$.
            We identify $S_k\cong\aff^k$ by projection to the first $k$ coordinates.
            We denote the strata inclusions by $\jmath_k\colon\aff^k\to\proj^n$, and write $\imath_k\colon\proj^k\hookto\proj^n$.
            By slight abuse of notation, we also write $\jmath_k$ and $\imath_k$ for the inclusions of $\aff^k$ and $\proj^k$, respectively, into any $\proj^l$ with $k\leq l\leq n$.
            
            We denote by $\Db(\mathrm{Sh}(\proj^n))$ the bounded derived category of sheaves of finite-dimensional $\kk$-vector spaces on $\proj^n$.
            The \emph{constructible derived category} is the full triangulated subcategory $\Dbc(\proj^n)\subset\Db(\mathrm{Sh}(\proj^n))$ consisting of the complexes whose cohomologies are (locally) constant on all strata.
            For brevity we write $\Hom_{\proj^n}(-,-)=\Hom_{\Dbc(\proj^n)}(-,-)$.
            We emphasize that we always consider the constructible derived category with respect to a fixed stratification, in constrast to e.g.~\cite{Achar}.

            The constructible derived category has a natural $t$-structure, which is obtained by iterated gluing of the shifted (by $-\frac{1}{2}\dim S_k=-k$) standard $t$-structures on $\Dbc(S_k)$ along the recollements provided by the strata inclusions, see (the proof of) \cite[Prop.~2.1.3]{BBD}.
            Its heart is the category of \emph{(middle-)perverse sheaves} $\perv{\proj^n}$.
            
            The \emph{Verdier duality functor} on $\Dbc(\proj^n)$ is $\verdier=\RHOM_{\Db(\proj^n)}(-,\omega_{\proj^n})$, where $\omega_{\proj^n}=a_{\proj^n}^!\ul{\kk}_\pt$ (with $a_{\proj^n}\colon \proj^n\to\pt$) is the dualizing sheaf.
            Since $\proj^n$ is smooth, we in fact have $\omega_{\proj^n}\cong\ul{\kk}_{\proj^n}[2n]$.
            Verdier duality restricts to a (contravariant) involution $\verdier\colon\Dbc(\proj^n)\to\Dbc(\proj^n)$.
            Moreover, from the definition of perverse sheaves it follows that $\verdier$ preserves $\perv{\proj^n}$.

            By variants of Beilinson's theorem (see \cite[Cor.~3.3.2]{MR1322847} or \cite[Prop.~1.5]{MR2119139}) the perverse $t$-structure on $\Dbc(\proj^n)$ has \emph{faithful heart}.
            This means that there is a \emph{realization functor}, i.e.~a triangulated functor  $\Db(\perv{\proj^n})\to\Dbc(\proj^n)$ such that the diagram
            \begin{equation*}
                \begin{tikzcd}
                    \perv{\proj^n}\arrow[hook]{d}\arrow[hook]{rd}&\\
                    \Db(\perv{\proj^n})\arrow{r}
                    &\Dbc(\proj^n)
                \end{tikzcd}
            \end{equation*}
            commutes, and that the realization functor is an equivalence.
            In particular, the realization functor provides isomorphisms $\Ext_{\perv{\proj^n}}^r(X,Y)\cong\Hom_{\proj^n}(X,Y[r])$ for all $r\geq 0$ (note that for $r\leq 1$ this holds even if the $t$-structure does not have faithful heart).
            For $r=1$ this allows to interpret triangles $Y\overset{f}{\to} Z\overset{g}{\to} X\to Y[1]$ with $X,Y,Z\in\perv{\proj^n}$ as short exact sequences $0\to Y\overset{f}{\to} Z\overset{g}{\to} X\to 0$ in $\perv{\proj^n}$.
        \subsection{Simple, standard and projective objects}\label{SimpleStdProj}
            We recall the explicit description of the simple perverse sheaves, and also the standard and costandard objects.
            This also allows one to describe the indecomposable projective and injective perverse sheaves, and derive some important  properties of $\perv{\proj^n}$.
            While all of this is well-known, see for instance \cite{Achar,BBD,MR1322847,CiprianiWoolf,MR0925719}, the explicit constructions in this subsection are central to the arguments in \cref{SerreFunctorDescr,PObjectClassification}.
            \subsubsection{Simple objects}
                The simple perverse sheaves are the \emph{IC sheaves} $\ic{k}=\jmath_{k,!*}\ul{\kk}_{\aff^k}[k]$ for $0\leq k \leq n$, where $\jmath_{k,!*}$ denotes the intermediate extension functor.
                As all strata closures are smooth, the IC sheaves are extensions by zero of shifted constant sheaves supported on the strata closures, i.e.~we have $\ic{k}\cong\imath_{k,*}\ul{\kk}_{\PP^k}[k]$.
                Note that $\verdier(\ic{k})\cong\ic{k}$, and $\imath_{k-1,*}\imath_{k-1}^*\ic{k}\cong\ic{k-1}[1]$ and $\imath_{k-1,*}\imath_{k-1}^!\ic{k}\cong\ic{k-1}[-1]$.
            \subsubsection{Standard objects}
                For $0\leq k\leq n$, the \emph{standard objects} $\stan{k}=\jmath_{k,!}\ul{\kk}_{\aff^k}[k]$ are perverse sheaves.
                We have $\stan{0}\cong\ic{0}$, and for $k\geq 1$ the recollements provide triangles
                \begin{equation}\label{delta_comp_series}
                    \ic{k-1}\longto{\phi^0_{k-1,k}}\stan{k}\longto{\mu_{k,k}}\ic{k}\longto{\epsilon^1_{k,k-1}}\ic{k-1}[1],
                \end{equation}
                where $\epsilon^1_{k,k-1}\colon \ic{k}\to\imath_{k-1,*}\imath_{k-1}^*\ic{k}\cong\ic{k-1}[1]$ is the adjunction unit and $\mu_{k,k}\colon\stan{k}\cong \jmath_{k,!}\jmath_k^*\ic{k}\to\ic{k}$ the adjunction counit.
                Note that for the applications in \cref{SerreFunctorDescr,PObjectClassification} one has to be careful with the choice of the (co)units, see \cref{simplemorph} for details.
                The notation here is more complicated than necessary at this point, but chosen for consistency with \cref{stdsimplemorph,subsec_projective_injectives} below where we will describe more general morphisms $\phi^r_{k,l}\colon \ic{k}\to\stan{l}[r]$ and $\mu_{k,l}\colon\stan{k}\to\ic{l}[l-k]$.
            
                Interpreted as short exact sequences in $\perv{\proj^n}$, the triangles \cref{delta_comp_series} are the composition series of the standard objects.
            
                Dually, the \emph{costandard objects} $\costan{k}=\jmath_{k,*}\ul{\kk}_{\aff^k}[k]\cong \verdier(\stan{k})$ are also perverse sheaves.
                Explicitly, their composition series are given by $\costan{0}\cong\ic{0}$ and the recollement triangles
                \begin{equation*}
                    \ic{k}\to\costan{k}\to\ic{k-1}\to\ic{k}[1]
                \end{equation*}
                for $k>0$.
            \subsubsection{Projective objects}
                The category $\perv{\proj^n}$ has enough projectives.
                The explicit construction of the indecomposable projective objects from the proof of \cite[Thm.~4.6]{CiprianiWoolf} shows that $P_n\cong\stan{n}$, while for $k<n$ the object $P_k$ arises from a triangle
                \begin{equation}\label{delta_flags}
                    \stan{k+1} \to P_k\to \stan{k}\longto{\delta_{k,k+1}}\stan{k+1}[1].
                \end{equation}
                In particular the indecomposable projective objects have $\stan{}$-flags.
                The morphism $\delta_{k,k+1}$ will be described precisely after \cref{DeltaHoms}.
                
                Dually, $\perv{\proj^n}$ also has enough injectives, and the indecomposable injective objects have $\costan{}$-flags given by $I_n=\costan{n}$ and the triangles
                \[
                    \costan{k} \to I_k\to \costan{k+1}\to \costan{k}[1]
                \]
                for $k<n$.
            \subsubsection{Highest weight structure}
                It follows that $\perv{\proj^n}$ is a highest weight category in the sense of \cite{ClineParshallScott} (see \cite{BrundanStroppelHW} for a modern treatment) with respect to the order $0<1<\dots<n$, since the (co)standard objects as defined above are indeed the (co)standard objects in the sense of highest weight categories: the triangles defining the indecomposable projective objects and the composition series of the standard objects show that $\stan{k}$ is the maximal a quotient of $P_k$ such that all composition factors except the top are $\ic{l}$ with $l<k$.
    
                By \cite[Prop.~A2.3]{DonkinqSchur} or \cite[Rem.~3.28]{BrundanStroppelHW}, the highest weight structure of $\perv{\proj^n}$ gives a bound on the global dimension, namely $\gdim{\perv{\proj^n}}\leq 2n$.
                Since the perverse $t$-structure has faithful heart, from  \cref{maps_between_simples} below we get $\Ext_{\perv{\proj^n}}^{2n}(\ic{n},\ic{n})\cong\Hom_{\proj^n}(\ic{n},\ic{n}[2n])\cong\kk$, and thus $\gdim{\perv{\proj^n}}=2n$.
        \subsection{Other descriptions of the category of perverse sheaves}\label{otherperspectives}
            There are two other well-known descriptions of the category $\perv{\proj^n}$, namely in terms of finite-dimensional algebras and via Lie algebras.
            We will not use these descriptions throughout the paper, with the exception that the classification results in \cref{sphericalexceptionals,onlyPlikes} rely on the classification of indecomposable perverse sheaves, which is obtained from the finite-dimensional algebras description.
    
            The Lie-theoretic description of the constructible derived category is given by the equivalence $\Dbc(\proj^n)\cong\Db(\catO^\lie{p}_0(\lie{sl}_{n+1}(\kk)))$ from \cite[Thm.~3.5.3]{MR1322847}, see also \cite[Rem.~7.3.10]{Achar} for an overview.
            Here $\lie{p}\subseteq\lie{sl}_{n+1}(\kk)$ is the maximal parabolic subalgebra with block sizes $(n,1)$.
            This equivalence identifies the standard $t$-structure on $\Db(\catO^\lie{p}_0(\lie{sl}_{n+1}(\kk)))$ with the perverse $t$-structure on $\Dbc(\proj^n)$, and thus yields $\perv{\proj^n}\cong\catO^\lie{p}_0(\lie{sl}_{n+1}(\kk))$.
            
            Furthermore, there is an equivalence  $\perv{\proj^n}\cong A_n\lmodfd$, where
            \begin{equation*}
                A_n=\kk\left(\begin{tikzcd}[bend angle=20]
                    \makebox(0.6cm,0.3cm)[c]{$0$}\arrow[bend left]{r}{b_1}&
                    \makebox(0.6cm,0.3cm)[c]{$1$}\arrow[bend left]{l}{a_1}\arrow[bend left]{r}{b_2}&
                    \makebox(0.6cm,0.3cm)[c]{$\dots$}\arrow[bend left]{l}{a_2}\arrow[bend left]{r}{b_{n-1}}&
                    \makebox(0.6cm,0.3cm)[c]{$n-1$}\arrow[bend left]{l}{a_{n-1}}\arrow[bend left]{r}{b_n}&
                    \makebox(0.6cm,0.3cm)[c]{$n$}\arrow[bend left]{l}{a_n}
                \end{tikzcd}\right)/
                \left(\begin{smallmatrix}
                    a_{i-1}a_i\\
                    b_{i+1}b_i\\
                    a_ib_i-b_{i-1}a_{i-1}\\
                    b_na_n
                \end{smallmatrix}\right).
            \end{equation*}
            The algebra $A_n$ is directly linked to the Lie-theoretic description by an equivalence $\catO^\lie{p}_0(\lie{sl}_{n+1})\cong A_n\lmodfd$, see \cite[Prop.~2.9]{MR1862802} and \cite[Ex.~1.1]{stroppelTQFT}.
            
            From the finite-dimensional algebras description, it is very easy to check the properties mentioned in \cref{SimpleStdProj}: for instance one can easily write down the indecomposable projective objects and the standard objects to see that $A_n\lmodfd$ is highest weight.
            One can also explicitly determine projective and injective resolutions of the simple $A_n$-modules, which in particular shows that $\gdim{A_n}=2n$.
            
            The algebra $A_n$ is special biserial, and thus \cite[p.~161, Thm.]{ButlerRingel} and \cite[Prop.~2.3]{WaldWaschbuesch} provide a combinatorial description of the indecomposable $A_n$-modules.
            Explicitly, there are the indecomposable projective-injective objects $P_i$ for $0\leq i\leq n-1$, and certain \emph{string modules} $\zz{\pm}{a}{b}$ with $0\leq b\leq a\leq n$, see \cite[\S 2.4]{MR4090927} or \cite[\S 4]{CiprianiLanini} for an explicit list.
            It follows that $A_n$ has $n+(n+1)+2\binom{n+1}{2}=n+(n+1)^2$ isomorphism classes of indecomposable modules.
            In \cref{strings} below we provide a construction of the string modules in terms of perverse sheaves, which can also be found in \cite{CiprianiLanini}.
            
            Note that it is a special property of $\proj^n$ that the indecomposable perverse sheaves can be classified.
            For more general (partial) flag varieties, the category of perverse sheaves is usually of wild representation type.
    \section{Description of the Serre functor of \texorpdfstring{$\Dbc(\PP^n)$}{Dbc(Pn)}} \label{SerreFunctorDescr}
        In this section we show that the IC sheaf $\ic{n}$ corresponding to the open stratum is a $\proj^n$-object in $\Dbc(\proj^n)$, and that the $\proj$-twist at $\ic{n}$ is the inverse Serre functor of $\Dbc(\proj^n)$.
        To do this, we first need to understand morphisms between simple perverse sheaves, and morphisms from simple to projective perverse sheaves.
        These technical results will also be used in \cref{PObjectClassification}.
        \subsection{Morphisms between simples}\label{simplemorph}
            Recall  (see e.g.~\cite[Rem.~1.2.5]{Achar}) that for a variety $X$ we have by definition $\Hom_X(\ul{\kk}_X,\ul{\kk}_X[r])=H^r(X)$.
            This allows to determine the morphisms between the shifted simple perverse sheaves in terms of the stratification.
            We briefly recall this well-known fact, see for instance \cite[p.~217, after the Remark]{MR1862802}.
            \begin{lemma}\label{maps_between_simples}
                For $0\leq k,l\leq n$ and $r\geq0$, there is an isomorphism of vector spaces
                \[
                    \Hom_{\proj^n}(\ic{k},\ic{l}[r])\cong H^{r-|k-l|}(\PP^l \cap \PP^k)
                \]
                In particular, if $k=l$ there is an isomorphism of graded algebras $\End_{\PP^n}^*(\ic{k})=H^*(\proj^k)\cong\kk[t]/(t^{k+1})$ with $\deg(t)=2$.
            \end{lemma}
            \begin{proof}
                This is obvious from $\imath_l^*\ic{k}\cong\ul{\kk}_{\proj^l}[k]$ for $k\geq l$, respectively $\imath_k^!\ic{l}\cong\ul{\kk}_{\proj^k}[2k-l]$ for $l\geq k$, and the definition of cohomology.
            \end{proof}
            For the rest of the paper we want to fix non-zero morphisms $\epsilon^r_{k,l}\colon \ic{k}\to\ic{l}[r]$ in a way that is compatible with composition.
            This requires us to inductively fix the (co)units for the recollement adjunctions, as follows.

            Suppose we have already fixed the adjunction (co)units $\eta_l\colon \id_{\Dbc(\proj^{l+1})}\to\imath_{l,*}\imath_l^*$ and $\varepsilon_l\colon \imath_{l,*}\imath_l^!\to\id_{\Dbc(\proj^{l+1})}$ for the recollements corresponding to the strata $S^{l+1}$ with $l+1\leq k<n$, such that $\eta_l\varepsilon_l=\varepsilon_{l-1}\eta_{l-1}\colon \ic{l}\to\ic{l}[2]$ and $\verdier(\varepsilon_l)_X=(\eta_l)_{\verdier(X)}$ for $X\in\Dbc(\proj^{l+1})$, and that the composition $\ic{k}\longto{\eta_{k-1}}\ic{k-1}[1]\longto{\varepsilon_{k-1}}\ic{k}[2]$ is non-zero.
            
            The inclusion of the $(k+1)$-dimensional stratum $S_{k+1}\cong\aff^{k+1}$ provides the recollement
            \begin{equation*}
                \begin{tikzcd}[column sep=huge]
					\Dbc(\proj^k)\arrow{r}{\imath_{k,*}}
					&\Dbc(\proj^{k+1})\arrow{r}{\jmath_{k+1}^*}\arrow[bend right]{l}[swap]{\imath_k^*}\arrow[bend left]{l}{\imath_k^!}
					&\Dbc(\aff^{k+1})\arrow[bend right]{l}[swap]{\jmath_{k+1,!}}\arrow[bend left]{l}{\jmath_{k+1,*}}
				\end{tikzcd}
            \end{equation*}
            As part of the recollement data, there are the adjunction unit $\tilde{\eta}_k\colon \id_{\Dbc(\proj^{k+1})}\to \imath_{k,*}\imath_k^*$ and the counit $\tilde{\varepsilon}_k\colon \imath_{k,*}\imath_k^!\to\id_{\Dbc(\proj^{k+1})}$.
            These can be chosen to be Verdier-dual to each other in the sense that $\verdier((\tilde{\varepsilon}_k)_{\verdier(X)})=(\tilde{\eta}_k)_X$ for $X\in\Dbc(\proj^{k+1})$.

            The proof of the following lemma is based on conversations with Jon Woolf.
            \begin{lemma}\label{counitcomp} Let $0\leq k\leq n-1$.\leavevmode
                \begin{enumerate}
                    \item The composition $\ic{k}\longto{\tilde{\varepsilon}_k}\ic{k+1}[1]\longto{\tilde{\eta}_k}\ic{k}[2]$ is non-zero for $k>0$.
                    \item The composition $\ic{k+1}\longto{\tilde{\eta}_k}\ic{k}[1]\longto{\tilde{\varepsilon}_k}\ic{k+1}[2]$
                        is non-zero.
                \end{enumerate}
            \end{lemma}
            \begin{proof}
                To see that the compositions are non-zero, we apply $a_{\proj^{k+1},*}$ to them, where $a_{\proj^{k+1}}\colon \proj^{k+1}\to\pt$.
                We have $a_{\proj^{k+1},*}\ic{k}\cong \bigoplus_{i=0}^k \ul{\kk}_\pt[-k+2i]$ since
                \begin{equation*}
                    \begin{split}
                        \Hom_\pt(\ul{\kk}_\pt,a_{\proj^{k+1},*}\ic{k}[r])&\cong\Hom_{\proj^{k+1}}(\ul{\kk}_{\proj^{k+1}},\imath_{k,*}\ul{\kk}_{\proj^k}[k+r])\\
                        &\cong\begin{cases}
                            \kk&\text{if $0\leq k+r\leq 2k$ and $k+r$ even},\\
                            0&\text{else},
                        \end{cases}
                    \end{split}
                \end{equation*}
                and analogously we get $a_{\proj^{k+1},*}\ic{k+1}\cong\bigoplus_{i=0}^{k+1} \ul{\kk}_\pt[-k-1+2i]$.
                Similar computations show that $a_{\proj^{k+1},*}\costan{k+1}\cong\ul{\kk}_\pt[k+1]$ and $a_{\proj^{k+1},*}\stan{k+1}\cong\ul{\kk}_\pt[-k-1]$ (for the latter, use $a_{\proj^{k+1},*}=a_{\proj^{k+1},!}$).
                
                Hence applying $a_{\proj^{k+1},*}$ to the triangles defining $\costan{k+1}[1]$ and $\stan{k+1}[2]$ yields triangles
                \begin{gather*}
                    \bigoplus_{i=0}^k \ul{\kk}_\pt[-k+2i]\longto{a_{\proj^{k+1},*}(\tilde{\varepsilon}_k)} \bigoplus_{i=0}^{k+1} \ul{\kk}_\pt[-k+2i]\to\ul{\kk}_\pt[k+2]\to\bigoplus_{i=0}^k \ul{\kk}_\pt[-k+2i+1],\\
                    \bigoplus_{i=0}^{k+1} \ul{\kk}_\pt[-k+2i]\longto{a_{\proj^{k+1},*}(\tilde{\eta}_k)}\bigoplus_{i=0}^k \ul{\kk}_\pt[-k+2i+2]\to\ul{\kk}_\pt[1-k]\to \bigoplus_{i=0}^{k+1} \ul{\kk}_\pt[-k+2i+1].
                \end{gather*}
                Since $\Hom_\pt(\ul{\kk}_\pt,\ul{\kk}_\pt[r])=0$ for $r\neq 0$, the morphisms $a_{\proj^{k+1},*}(\tilde{\varepsilon}_k)$ and $a_{\proj^{k+1},*}(\tilde{\eta}_k)$ must identify all the summands occuring in both their source and target.
                
                It follows that the composition $a_{\proj^{k+1},*}(\ic{k}\longto{\tilde{\varepsilon}_k}\ic{k+1}[1]\longto{\tilde{\eta}_k}\ic{k}[2])$ identifies all the summands occuring in both $a_{\proj^{k+1},*}\ic{k}$ and $a_{\proj^{k+1},*}\ic{k}[2]$, and such a common summand exists if and only if $k>0$.
                
                For the other composition, one sees similarly that $a_{\proj^{k+1},*}(\ic{k+1}\longto{\tilde{\eta}_k}\ic{k}[1]\longto{\tilde{\varepsilon}_k}\ic{k+1}[2])$ identifies all the summands occuring in both $a_{\proj^{k+1},*}\ic{k+1}$ and $a_{\proj^{k+1},*}\ic{k+1}[2]$, and thus is non-zero.\qedhere
            \end{proof}
            In particular, since $\Hom_{\proj^n}(\ic{k},\ic{k}[2])$ is $1$-dimensional by \cref{maps_between_simples}, we have $\tilde{\eta}_k\tilde{\varepsilon}_k=\lambda\varepsilon_{k-1}\eta_{k-1}$ for some scalar $\lambda\neq 0$.
            Set $\eta_k=\frac{1}{\sqrt{\lambda}}\tilde{\eta}_k$ and $\varepsilon_k=\frac{1}{\sqrt{\lambda}}\tilde{\varepsilon}_k$.
            These are again adjunction (co)units for the recollement for gluing the ($k+1$)-dimensional stratum, and rescaling both by the same factor ensures that the new (co)units are still Verdier-dual to each other.
            Note that to ensure that the recollement triangles are isomorphic to those obtained from the original (co)units, one also has to rescale the (co)units for the adjunctions between $\jmath_{k+1,!}$, $\jmath_{k+1}^*$ and $\jmath_{k+1,*}$.
            
            By construction, the square
            \begin{equation}
                \label{coherencerelation}
                \begin{tikzcd}
                    \ic{k}\arrow{r}{\eta_{k-1}}\arrow{d}[swap]{\varepsilon_k}
                    &\ic{k-1}[1]\arrow{d}{\varepsilon_{k-1}}\\
                    \ic{k+1}[1]\arrow{r}{\eta_k}
                    &\ic{k}[2]
                \end{tikzcd}
            \end{equation}
            now commutes, as desired.
            Moreover, the composition $\ic{k}\longto{\varepsilon_k}\ic{k+1}[1]\longto{\eta_k}\ic{k}[2]$ is still non-zero, which completes the induction.

            From now on, we assume the (co)units $\eta_k$, $\varepsilon_k$ of the recollement adjunctions are fixed for all $k$, and make the square \cref{coherencerelation} commute.
            With this data, we can now fix the desired morphisms $\epsilon^r_{k,l}\colon \ic{k}\to\ic{l}[r]$ spanning $\Hom_{\proj^n}(\ic{k},\ic{l}[r])\neq 0$ (with $|k-l|\leq r\leq |k-l|+2\min(k,l)$ and $r-|k-l|$ even), as follows.
            
            Set $\epsilon^1_{k,k-1}=\eta_{k-1}\colon \ic{k}\to\ic{k-1}[1]$ and $\epsilon^1_{k,k+1}=\varepsilon_k\colon\ic{k}\to\ic{k+1}[1]$.
            Furthermore, for $k>0$ let $\epsilon^2_{k,k}=\epsilon^1_{k-1,k}\epsilon^1_{k,k-1}$, which is a generator of $\End_{\proj^n}(\ic{k})$ by \cref{maps_between_simples,counitcomp}.
            The proof of \cref{maps_between_simples} shows that $\Hom_{\proj^n}(\ic{k},\ic{l}[r])$ is spanned by the composition
            \begin{equation*}
                \ic{k}\longto{\epsilon^1_{k,k-1}}\ic{k-1}[1]\longto{\epsilon^1_{k-1,k-2}}\dots\longto{\epsilon^1_{l+1,l}}\ic{l}[k-l]\longto{(\epsilon^2_{l,l})^{(r-k+l)/2}}\ic{l}[r]
            \end{equation*}
            for $l\leq k$, and by 
            \begin{equation*}
                \ic{k}\longto{(\epsilon^2_{k,k})^{(r+k-l)/2}}\ic{k}[r+k-l]\longto{\epsilon^1_{k,k+1}}\ic{k+1}[r+k-l+1]\longto{\epsilon^1_{k+1,k+2}}\dots\longto{\epsilon^1_{l-1,l}}\ic{l}[r]
            \end{equation*}
            for $l\geq k$.
            We write $\epsilon^r_{k,l}\colon\ic{k}\to\ic{l}[r]$ for these compositions.
            Note that $\verdier(\epsilon^r_{k,l})=\epsilon^r_{l,k}$ since $\verdier(\epsilon^1_{k,k-1})=\epsilon^1_{k-1,k}$.
            
            As an immediate consequence of the relation \cref{coherencerelation}, in fact any non-zero composition of the morphisms $\epsilon^1_{i,i\pm 1}$ between $\ic{k}$ and $\ic{l}[r]$ yields the same morphism.
            This shows that $\epsilon^s_{l,m}\epsilon^r_{k,l}=\epsilon^{r+s}_{k,m}$ if $\Hom_{\proj^n}(\ic{k},\ic{m}[r+s])\neq 0$.
            
            The above construction also yields the following explicit description of the algebra $\bigoplus_{k,l}\Hom_{\proj^n}^*(\ic{k},\ic{l})$, as obtained in \cite[Proof of Prop.~2.9]{MR1862802}.
            \begin{proposition}
                \label{ICextalgebra}
                There is an isomorphism of graded algebras $\bigoplus_{k,l}\Hom_{\proj^n}^*(\ic{k},\ic{l})\cong E_n$, where
                \begin{equation*}
                    E_n=\kk\left(\begin{tikzcd}[bend angle=20]
                        \makebox(0.6cm,0.3cm)[c]{$0$}\arrow[bend left]{r}{\epsilon_{0,1}}&
                        \makebox(0.6cm,0.3cm)[c]{$1$}\arrow[bend left]{r}{\epsilon_{1,2}}\arrow[bend left]{l}{\epsilon_{1,0}}&
                        \makebox(0.6cm,0.3cm)[c]{$\dots$}\arrow[bend left]{r}{\epsilon_{n-2,n-1}}\arrow[bend left]{l}{\epsilon_{2,1}}&
                        \makebox(0.6cm,0.3cm)[c]{$n-1$}\arrow[bend left]{r}{\epsilon_{n-1,n}}\arrow[bend left]{l}{\epsilon_{n-1,n-2}}&
                        \makebox(0.6cm,0.3cm)[c]{$n$}\arrow[bend left]{l}{\epsilon_{n,n-1}}
                    \end{tikzcd}\right)/
                    \left(\begin{smallmatrix}
                        \epsilon_{1,0}\epsilon_{0,1}\\
                        \epsilon_{k+1,k}\epsilon_{k,k+1}-\epsilon_{k-1,k}\epsilon_{k,k-1}
                    \end{smallmatrix}\right)
                \end{equation*}
                with $\deg(\epsilon_{k,k\pm 1})=1$.
            \end{proposition}
            \begin{proof}
                The isomorphism $E_n\to\bigoplus_{k,l}\Hom_{\proj^n}^*(\ic{k},\ic{l})$ is given by $\epsilon_{k,k\pm 1}\mapsto\epsilon^1_{k,k\pm 1}$, which is well-defined since $\Hom_{\proj^n}(\ic{0},\ic{0}[2])=0$ and the square \cref{coherencerelation} commutes for all $k$.
                Surjectivity follows from the proof of \cref{maps_between_simples} and \cref{counitcomp}.
                One then easily checks that the graded dimensions of both algebras agree.
            \end{proof}
            \begin{remark}\label{ICdegree2end}
                For $0<k<n$ we moreover have $\epsilon^2_{k,k}=\imath_k^!(\epsilon^1_{k+1,k})$.
                Indeed, naturality of the recollement triangles yields the commutative diagram
                \begin{equation*}
                    \begin{tikzcd}
                        \ic{k}[-1]\arrow{r}{\epsilon^1_{k,k+1}}\arrow{d}[swap]{\imath_k^!(\epsilon^1_{k+1,k})}
                        &\ic{k+1}\arrow{r}\arrow{d}{\epsilon^1_{k+1,k}}
                        &\costan{k+1}\arrow{r}\arrow{d}
                        &\ic{k}\arrow{d}\\
                        \ic{k}[1]\arrow{r}{\id}
                        &\ic{k}[1]\arrow{r}
                        &0\arrow{r}
                        &\ic{k}[2]\rlap{.}
                    \end{tikzcd}
                \end{equation*}
                From this it also follows that $\epsilon^2_{k,k}=\imath_k^!(\epsilon^2_{k+1,k+1})$:
                we have $\epsilon^2_{k,k}=\epsilon^1_{k+1,k}\epsilon^1_{k,k+1}$ and $\imath_k^!(\epsilon^1_{k,k+1})=\id$, as $\epsilon^1_{k,k+1}$ is the counit for the adjunction between $\imath_{k,*}$ and $\imath_k^!$.
            \end{remark}
        \subsection{Morphisms between standards and simples}\label{stdsimplemorph}
            Morphisms from the standard objects to IC sheaves, and dually from IC sheaves to costandard objects, are easily calculated by the recollement adjunctions:
            \begin{lemma}\label{topIC_to_lower_nabla}
                We have
                \begin{equation*}
                   \Hom_{\proj^n}(\stan{k},\ic{l}[r])\cong\Hom_{\proj^n}(\ic{l},\costan{k}[r])\cong\begin{cases}
                        \kk&\text{if $l\geq k$ and $r=l-k$},\\
                        0&\text{else}.
                    \end{cases}
                \end{equation*}
            \end{lemma}
            \begin{proof}
                By Verdier duality it suffices to compute $\Hom_{\proj^n}(\stan{k},\ic{l}[r])$.
                For $l\geq k$ we have
                \begin{equation*}
                    \begin{split}
                        \Hom_{\proj^n}(\stan{k},\ic{l}[r])&\cong\Hom_{\proj^k}(\jmath_{k,!}\ul{\kk}_{\aff^k}[k],\ul{\kk}_{\proj^k}[2k+r-l])\\
                        &\cong\Hom_{\aff^k}(\ul{\kk}_{\aff^k},\ul{\kk}_{\aff^k}[k+r-l]),
                    \end{split}
                \end{equation*}
                and from this the claim follows.
                The case $l<k$ is similar but easier, as $\jmath_k^*(\ic{l})=0$.
            \end{proof}
            By the proof of \cref{topIC_to_lower_nabla}, for $l\geq k$ there is a canonical non-zero morphism $\mu_{k,l}\colon\stan{k}\to\ic{l}[l-k]$ corresponding to $\id_{\aff^k}$ under the adjunctions.
            In particular, for $k=l$ these are the (co)units from the recollement triangles \cref{delta_comp_series} defining $\stan{k}$.
            Moreover, the proof also shows that $\mu_{k,l}$ is the unique morphisms making the diagram
            \begin{equation*}
                \begin{tikzcd}
                    \stan{k}\arrow{d}[swap]{\mu_{k,k}}\arrow[dashed]{rd}{\mu_{k,l}}&\\
                    \ic{k}\arrow{r}{\epsilon^{l-k}_{k,l}}
                    &\ic{l}[l-k]
                \end{tikzcd}
            \end{equation*}
            commute.
            \begin{lemma}\label{DeltaHoms}
                For $0\leq k,l\leq n$ we have
                \begin{equation*}
                    \Hom_{\proj^n}(\stan{k},\stan{l}[r])\cong\Hom_{\proj^n}(\costan{l},\costan{k}[r])\cong\begin{cases}
                        \kk&\text{if $l>k$ and $r\in\{l-k-1,l-k\}$},\\
                        \kk&\text{if $l=k$ and $r=0$},\\
                        0&\text{otherwise}.
                    \end{cases}
                \end{equation*}
            \end{lemma}
            \begin{proof}
                By Verdier duality it suffices to compute $\Hom_{\proj^n}(\stan{k},\stan{l}[r])$.
                For this, the claim follows from \cref{topIC_to_lower_nabla} and the long exact sequence obtained by applying $\Hom_{\proj^n}(\stan{k},-)$ to the triangle \cref{delta_comp_series} defining $\stan{l}$.
            \end{proof}
            The proof of \cref{DeltaHoms} also yields the following descriptions of canonical morphisms $\delta^r_{k,l}\colon\stan{k}\to\stan{l}[r]$ spanning $\Hom_{\proj^n}(\stan{k},\stan{l}[r])$:
            \begin{itemize}
                \item For $r=l-k$, we define $\delta^{l-k}_{k,l}\colon\stan{k}\to\stan{l}[l-k]$ as the unique morphism making the diagram
                    \begin{equation*}
                        \begin{tikzcd}
                            \stan{k}\arrow[dashed]{r}{\delta^{l-k}_{k,l}}\arrow{rd}[swap]{\mu_{k,l}}
                            &\stan{l}[l-k]\arrow{d}{\mu_{l,l}}\\
                            &\ic{l}[l-k]
                        \end{tikzcd}
                    \end{equation*}
                    commute.
                \item For $r=l-k-1$, we define $\delta^{l-k-1}_{k,l}\colon\stan{k}\to\stan{l}[l-k-1]$ as the composition
                    \begin{equation*}
                        \stan{k}\longto{\mu_{k,l-1}}\ic{l-1}[l-1-k]\longto{\phi^0_{l-1,l}}\stan{l}[l-1-k],
                    \end{equation*}
                    where $\phi^0_{l-1,l}\colon \ic{l-1}\to\stan{l}$ is the morphism from the recollement triangle \cref{delta_comp_series} defining $\stan{l}$.
            \end{itemize}
            In particular, the morphisms $\delta_{k,k+1}\colon \stan{k}\to\stan{k+1}[1]$ can be used to define the indecomposable projective objects via the triangles \cref{delta_flags}.
        \subsection{Morphisms between simples and projectives}\label{subsec_projective_injectives}
            We start by computing the morphisms from simple perverse sheaves to standard objects.
            \begin{lemma}\label{mor_simple_stand}
            For $0\leq k,l\leq n$ and $r\in\Z$ we have
                \[
                    \Hom_{\proj^n}(\ic{l},\stan{k}[r])\cong \Hom_{\proj^n}(\costan{k},\ic{l}[r])\cong \begin{cases}
                        \kk&\text{if $r=k+l$},\\
                        \kk&\text{if $r=k-l-1$ and $l<k$},\\
                        0&\text{else}.
                    \end{cases}
                \]
            \end{lemma}
            \begin{proof} 
            By Verdier duality it suffices to compute $\Hom_{\proj^n}(\ic{l},\stan{k}[r])$. Applying the functor $\Hom_{\proj^n}(\ic{l},-)$ to \cref{delta_comp_series} yields a long exact sequence
                \[
                    \ldots\to \Hom_{\proj^n}(\ic{l},\ic{k-1}[r]) \to \Hom_{\proj^n}(\ic{l},\stan{k}[r]) \to \Hom_{\proj^n}(\ic{l},\ic{k}[r])\to\ldots
                \]
                in which the left and right-hand side can be computed by \cref{maps_between_simples}, and these Hom spaces are spanned by $\epsilon^r_{l,k-1}$ and $\epsilon^r_{l,k}$, respectively.
                Since $\epsilon^{r+1}_{l,k-1}=\epsilon^1_{k,k-1}\epsilon^r_{l,k}$, the connecting morphisms $\Hom_{\proj^n}(\ic{l},\ic{k}[r])\to\Hom_{\proj^n}(\ic{l},\ic{k-1}[r+1])$ are isomorphisms unless $r=l+k$, or $l<k$ and $r=k-l-2$.
            \end{proof}
            The proof of \cref{mor_simple_stand} shows that $\Hom_{\proj^n}(\ic{l},\stan{k}[k+l])$ is spanned by the unique morphism $\phi^{l+k}_{l,k}\colon\ic{l}\to\stan{k}[k+l]$ making the diagram 
            \begin{equation*}
                \begin{tikzcd}
                    &\ic{l}\arrow{d}{\epsilon^{k+l}_{k,l}}\arrow[dashed]{ld}[swap]{\phi^{l+k}_{l,k}}\\
                    \stan{k}[k+l]\arrow{r}[swap]{\mu_{k,k}}
                    &\ic{k}[k+l]\rlap{.}
                \end{tikzcd}
            \end{equation*}
            commute.
            
            For $l<k$, the proof shows that $\Hom_{\proj^n}(\ic{l},\stan{k}[k-l-1])$ is spanned by the composition
            \begin{equation*}
                \phi^{k-l-1}_{l,k}\colon\ic{l}\longto{\epsilon^{k-l-1}_{l,k-1}}\ic{k-1}[k-l-1]\longto{\phi^0_{k-1,k}}\stan{k}[k-l-1],
            \end{equation*}
            where $\phi^0_{k-1,k}\colon \ic{k-1}\to\stan{k}$ is the morphism from the recollement triangle \cref{delta_comp_series}.
            The non-zero morphisms $\costan{k}\to\ic{l}[r]$ are defined by the dual diagrams.
            
            From \cref{mor_simple_stand} it also follows that all indecomposable projective perverse sheaves except $P_n$ are projective-injective:
            \begin{proposition}\label{proj_inj}
                For $0\leq k\leq n-1$ and $0\leq l\leq n$ we have
                \begin{equation*}
                    \Hom_{\proj^n}(\ic{l},P_k[r])\cong\begin{cases}
                        \kk&\text{if $l=k$ and $r=0$},\\
                        0&\text{else}.
                    \end{cases}
                \end{equation*}
                In particular, if $k<n$ then $P_k$ is the injective hull of $\ic{k}$ in $\perv{\proj^n}$.
            \end{proposition}
            \begin{proof}
                That $P_k$ is the injective hull of $\ic{k}$ in $\perv{\proj^n}$ is immediate from the first part, using $\Ext^1_{\perv{\proj^n}}(-,-)\cong\Hom_{\proj^n}(-,-[1])$.
                
                To compute $\Hom_{\proj^n}(\ic{l},P_k[r])$, we apply the functor $\Hom_{\proj^n}(\ic{l},- )$ to \cref{delta_flags} to get the long exact sequence
                \[
                    \ldots \to\Hom_{\proj^n}(\ic{l},\stan{k+1}[r])\to\Hom_{\proj^n}(\ic{l},P_k[r])\to\Hom_{\proj^n}(\ic{l},\stan{k}[r])\to\ldots
                \]
                We claim that $\delta_{k,k+1}\colon \stan{k}\to\stan{k+1}[r]$ induces isomorphisms in all degrees.
                \begin{enumerate}[label=Case~\arabic*:]
                    \item $l>k$.
                        By \cref{mor_simple_stand}, $\Hom_{\proj^n}(\ic{l},\stan{k+1}[r+1])$ and $\Hom_{\proj^n}(\ic{l},\stan{k}[r])$ are $1$-dimensional for $r=l+k$, and vanish otherwise.
                        By the construction of the morphisms spanning these Hom spaces, we have to show that in the diagram
                        \begin{equation*}
                            \begin{tikzcd}[column sep=small]
                                \stan{k}[k+l]\arrow[dashed]{rr}{\delta_{k,k+1}}\arrow{dd}[swap]{\mu_{k,k}}
                                &&\stan{k+1}[k+l+1]\arrow{dd}{\mu_{k+1,k+1}}\\
                                &\ic{l}\arrow{ld}[swap]{\epsilon^{k+l}_{l,k}}\arrow{rd}{\epsilon^{k+l+1}_{l,k+1}}\arrow[dashed]{lu}[swap]{\phi^{k+l}_{l,k}}\arrow[dashed]{ru}{\phi^{k+1+l}_{l,k+1}}\\
                                \ic{k}[k+l]\arrow{rr}[swap]{\epsilon^1_{k,k+1}}
                                &&\ic{k+1}[k+l+1]\rlap{.}
                            \end{tikzcd}
                        \end{equation*}
                        the triangle consisting of dashed arrows commutes.
                        By construction we have $\epsilon^{k+l+1}_{l,k+1}=\epsilon^1_{k,k+1}\epsilon^{k+l}_{l,k}$, and the outer square commutes by the definition of the morphism $\delta_{k,k+1}\colon\stan{k}\to\stan{k+1}[1]$.
                        An easy diagram chase then shows $\phi^{k+1+l}_{l,k+1}=\delta_{k,k+1}\phi^{k+l}_{l,k}$, as required.
                    \item $l=k$.
                        From \cref{mor_simple_stand} we know that $\Hom_{\proj^n}(\ic{l},\stan{k}[r])$ is $1$-dimensional for $r\in\{0,2k+1\}$ and $\Hom_{\proj^n}(\ic{l},\stan{k+1}[r+1])$ is $1$-dimensional for $r=2k$, and they vanish otherwise.
                        By the same argument as in Case~1, $\delta_{k,k+1}$ induces an isomorphism for $r=2k$, and it follows from the long exact sequence that $\Hom_{\proj^n}(\ic{k},P_k[r])$ is $1$-dimensional for $r=0$ and vanishes otherwise.
                    \item $l<k$.
                        From \cref{mor_simple_stand} we know that $\Hom_{\proj^n}(\ic{l},\stan{k}[r])$ is $1$-dimensional for $r\in\{k-l-1,k+l\}$ and $\Hom_{\proj^n}(\ic{l},\stan{k+1}[r+1])$ is $1$-dimensional for $r\in\{k-l,k+l+1\}$, and they vanish otherwise.
                        For $r=k+l$, that $\delta_{k,k+1}$ induces an isomorphism follows by the same argument as in Case~1.
    
                        For $r=k-l-1$, we need to show $\phi^{k-l}_{l,k+1}=\delta_{k,k+1}\phi^{k-l-1}_{l,k}$, which amounts to checking that the diagram
                        \begin{equation*}
                            \begin{tikzcd}[column sep=small]
                                \ic{k-1}[k-l-1]\arrow{rr}{\epsilon^1_{k-1,k}}\arrow{dd}[swap]{\phi^0_{k-1,k}}
                                &&\ic{k}[k-l]\arrow{dd}{\phi^0_{k,k+1}}\\
                                &\ic{l}\arrow{lu}[swap]{\epsilon^{k-l-1}_{l,k-1}}\arrow{ru}{\epsilon^{k-l}_{l,k}}\arrow[dashed]{ld}[swap]{\phi^{k-l-1}_{l,k}}\arrow[dashed]{rd}{\phi^{k-l}_{l,k+1}}&\\
                                \stan{k}[k-l-1]\arrow{rr}{\delta_{k,k+1}}
                                &&\stan{k+1}[k-l]
                            \end{tikzcd}
                        \end{equation*}
                        commutes.
                        By construction, we have $\epsilon^{k-l}_{l,k}=\epsilon^1_{k-1,k}\epsilon^{k-l-1}_{l,k-1}$, and the claim follows by a straightforward diagram chase provided the outer square commutes.
    
                        To see this, observe that $\epsilon^1_{k-1,k}\epsilon^1_{k,k-1}=\epsilon^1_{k+1,k}\epsilon^1_{k,k+1}$, the definition of $\delta_{k,k+1}$, and the axiom (TR3) give the commutative diagram
                        \begin{equation*}
                            \begin{tikzcd}[column sep=width("aaaaaaaa")]
                                \ic{k-1}\arrow{r}{\phi^0_{k-1,k}}\arrow{d}{\epsilon^1_{k-1,k}}
                                &\stan{k}\arrow{r}{\mu_{k,k}}\arrow{d}{\delta_{k,k+1}}
                                &\ic{k}\arrow{r}{\epsilon^1_{k,k-1}}\arrow{d}{\epsilon^1_{k,k+1}}
                                &\ic{k-1}[1]\arrow{d}{\epsilon^1_{k-1,k}}\\
                                \ic{k}[1]\arrow{r}{\phi^0_{k,k+1}}
                                &\stan{k+1}[1]\arrow{r}{\mu_{k+1,k+1}}
                                &\ic{k+1}[1]\arrow{r}{\epsilon^1_{k+1,k}}
                                &\ic{k}[2]\rlap{,}
                            \end{tikzcd}
                        \end{equation*}             
                       in which the left square is the desired commutative square (up to shift).\qedhere
                \end{enumerate}
            \end{proof}
        \subsection{\texorpdfstring{$\PP$}{P}-like simple perverse sheaves}
            By the explicit description of the simple perverse sheaves as constant sheaves on the stratum closures, it is obvious that they are $\PP^k$-like objects.
            However, only one of them is Calabi--Yau:
            \begin{proposition}\label{simples_P-obj}\leavevmode
                \begin{enumerate}
                    \item For $0\leq k\leq n$, the simple perverse sheaf $\ic{k}$ is a $\PP^k$-like object in $\Dbc(\proj^n)$.
                    \item The simple perverse sheaf $\ic{n}$ is a $\PP^n$-object in $\Dbc(\proj^n)$.
                    \item $\ic{k}$ is not Calabi--Yau in $\Dbc(\proj^n)$ if $k<n$.
                \end{enumerate}
            \end{proposition}
            \begin{proof}\leavevmode
                \begin{enumerate}
                    \item Immediate by \cref{maps_between_simples}.
                    \item Since $\ic{n}$ is $\PP^n$-like by the first part and $\Hom_{\proj^n}^*(X,Y)$ is finite-dimensional for any $X,Y\in\Dbc(\proj^n)$, we only need to check the Calabi--Yau property.
                        As the perverse $t$-structure has faithful heart, by \cref{check_CY_on_proj} it is enough to check that the composition pairing 
                        \begin{equation*}
                            \Hom_{\proj^n}(P,\ic{n}[r])\otimes \Hom_{\proj^n}(\ic{n},P[2n-r])\to\Hom_{\proj^n}(\ic{n},\ic{n}[2n])\cong\kk
                        \end{equation*}
                        is non-degenerate for any $r$ and any indecomposable projective object $P\in\perv{\PP^n}$.
        
                        For $P=P_k$ with $k<n$, we apply $\Hom_{\proj^n}(-,\ic{n})$ to \cref{delta_flags}.
                        From \cref{topIC_to_lower_nabla,DeltaHoms} it follows that all connecting morphisms are isomorphisms, so $\Hom_{\proj^n}(P,\ic{n}[r])=0$ for all $r$.
                        We also have $\Hom_{\proj^n}(\ic{n},P[2n-r])=0$ for all $r$ by \cref{proj_inj}, and thus the only non-trivial case is $P=P_n$.
                        Alternatively, for this one can also use that $P_k=I_k$ is the projective cover and injective hull of $\ic{k}$, and that the perverse $t$-structure has faithful heart.
                        
                        As $P_n=\stan{n}$, we know from \cref{topIC_to_lower_nabla} that $\Hom_{\proj^n}(P_n,\ic{n}[r])$ is one-dimensional if $r=0$, and vanishes otherwise.
                        From \cref{mor_simple_stand} we know that $\Hom_{\proj^n}(\ic{n},P_n[2n-r])$ is one-dimensional for $r=0$ and vanishes otherwise.
                        Moreover, by construction of $\phi^{2n}_{n,n}$ the composition $\ic{n}\longto{\phi^{2n}_{n,n}} \stan{n}[2n]\longto{\mu_{n,n}}\ic{n}[2n]$ is precisely $\epsilon^{2n}_{n,n}\colon \ic{n}\to\ic{n}[2n]$, and thus the composition pairing is non-degenerate.
                    \item For $k<n$, the composition pairing
                        \begin{equation*}
                            \Hom_{\proj^n}(P,\ic{k}[r])\otimes \Hom_{\proj^n}(\ic{k},P[2k-r])\to\Hom_{\proj^n}(\ic{k},\ic{k}[2k])\cong\kk
                        \end{equation*}
                        cannot be non-degenerate since for $P=P_k=I_k$ the tensor factors on the left-hand side are non-zero only for $r=0$ and $r=2k$, respectively.\qedhere
                \end{enumerate}
            \end{proof}
        \subsection{Characterization of the Serre functor}
            The following characterization of Serre functors is adapted from \cite[Thm.~3.4]{MR2369489}.
            The main difference is that we would like to start with a triangulated functor which looks like a derived functor, but is not a priori known to arise as a derived functor.
            Showing that such a functor is indeed a derived functor is hard if one only uses triangulated categories, see e.g.~\cite{RickardMO}.
            However, this technical issue can be resolved by using $\infty$-enhancements.
            \begin{lemma}\label{serrefunctorcharacterization}
                Let $\cat{A}$ be a finite-length abelian category with finitely many simples, enough projectives and enough injectives.
                Assume that $\cat{A}$ is of finite global dimension, all the projective-injective objects in $\cat{A}$ have isomorphic top and socle, and that there is a projective generator $P$ of $\cat{A}$ admitting a presentation $0\to P\to X_1\to X_2$ with $X_1$, $X_2$ projective-injective.
                Let $\D^+_\infty(\cat{A})$ be the derived category of $\cat{A}$ in the $\infty$-categorical sense as defined in \cite[Variant~1.3.2.8]{LurieHigherAlgebra}, and let $F\colon \D^+_\infty(\cat{A})\to\D^+_\infty(\cat{A})$ be a functor of $\infty$-categories.
                
                If the triangulated functor $hF\colon \D^+(\cat{A})\to\D^+(\cat{A})$ satisfies the conditions
                \begin{enumerate}
                    \item $hF$ restricts to an equivalence $hF\colon\D^b(\cat{A})\to\D^b(\cat{A})$,
                    \item $hF(\D^+(\cat{A})^{\geq 0})\subseteq\D^+(\cat{A})^{\geq 0}$, where $\D^+(\cat{A})^{\geq 0}$ denotes the non-negative part of the standard $t$-structure,
                    \item $hF(\Inj(\cat{A}))\subseteq\Proj(\cat{A})$,
                    \item $H^0\circ hF$ preserves the subcategory $\ProjInj(\cat{A})$ of projective-injective objects, and restricted to this category is isomorphic to the inverse Nakayama functor $\nu^{-1}$,
                \end{enumerate}
                then $hF\colon \Db(\cat{A})\to\Db(\cat{A})$ is an inverse Serre functor for $\Db(\cat{A})$.
            \end{lemma}
            \begin{proof}
                The argument essentially follows the proof of \cite[Thm.~3.4]{MR2369489}.
                We write $\mathbb{R}\nu^{-1}\colon \D^+_\infty(\cat{A})\to\D^+_\infty(\cat{A})$ for the right derived functor of the inverse Nakayama functor in the $\infty$-categorical sense, see \cite[Ex.~1.3.3.4]{LurieHigherAlgebra} for the definition (actually we use the dual version, obtained by $\D^+_\infty(\cat{A})=\D^-_\infty(\cat{A}^\op)^\op$).
                Then $h\mathbb{R}\nu^{-1}=\mathrm{R}\nu^{-1}\colon \D^+(\cat{A})\to\D^+(\cat{A})$ is the usual right derived functor, and its restriction $\mathrm{R}\nu^{-1}\colon \Db(\cat{A})\to\Db(\cat{A})$ is the inverse Serre functor by \cref{fdalgserrefunctor}.
                We show that $\mathbb{R}\nu^{-1}\cong F$, which then implies the claim.
                \begin{steps}
                    \item On the subcategory $\Proj(\cat{A})$, we have $H^0\circ hF\cong\nu^{-1}$.
        
                        Proof: By assumption, the projective generator $P$ of $\cat{A}$ admits a presentation $0\to P\to X_1\to X_2$ with $X_1$, $X_2$ projective-injective.
                        Since $H^0\circ hF$ and $\nu^{-1}$ are left exact, and $H^0\circ hF\cong\nu^{-1}$ on the subcategory $\ProjInj(\cat{A})$, it follows that $H^0(F(P))\cong\nu^{-1}(P)$.
                        It is easy to see that this isomorphism is functorial in $P$ and compatible with taking direct sums and summands, which proves the claim.
                    \item $hF\colon \Db(\cat{A})\to\Db(\cat{A})$ commutes with the (inverse) Serre functor.
                    
                        Proof: From the Yoneda lemma it follows that the Serre functor commutes with autoequivalences.
                    \item On the subcategory $\Inj(\cat{A})$, we have $H^0\circ hF\circ\nu^{-1}\cong\nu^{-1}\circ H^0\circ hF$.
        
                        Proof: By assumption we have $hF\cong H^0\circ hF$ on $\Inj(\cat{A})$.
                        With this and Step~2 we get
                        \begin{equation*}
                            \Serre^{-1} \circ H^0\circ hF\cong\Serre^{-1}\circ hF\cong hF\circ \Serre^{-1}\cong hF\circ\nu^{-1}.
                        \end{equation*}
                        Observe that $H^0\circ hF$ takes $\Inj(\cat{A})$ to $\cat{A}$, and therefore taking $H^0$ on both sides yields
                        \begin{equation*}
                            \nu^{-1}\circ H^0\circ hF\cong H^0\circ\Serre^{-1}\circ H^0\circ hF\cong H^0\circ hF\circ\nu^{-1}.
                        \end{equation*}
                    \item $H^0\circ hF$ is fully faithful on the subcategory $\Proj(\cat{A})$.

                        Proof: By assumption, $H^0\circ hF$ is isomorphic to $\nu^{-1}$ on the full subcategory $\ProjInj(\cat{A})$, and $\nu^{-1}\colon\ProjInj(\cat{A})\to\ProjInj(\cat{A})$ is an autoequivalence.
                        Let $\cat{C}\subseteq\cat{A}$ be the full subcategory of objects $M$ admitting a presentation $0\to M\to X_1\to X_2$ with $X_1$, $X_2$ projective-injective.
                        Then $H^0\circ hF$ restricts\slash extends to $H^0\circ hF\colon \cat{C}\to\cat{C}$.
                        Moreover, $(H^0\circ hF)^{-1}$ can be extended to $(H^0\circ hF)^{-1}\colon\cat{C}\to\cat{C}$ by setting $(H^0\circ hF)^{-1}(\ker(\phi\colon X_1\to X_2))=\ker((H^0\circ hF)^{-1}(\phi)\colon (H^0\circ hF)^{-1}(X_1)\to (H^0\circ hF)^{-1}(X_2))$.
                        By assumption all projective objects lie in $\cat{C}$, and thus $H^0\circ hF$ is fully faithful on $\Proj(\cat{A})$.
                    \item We have $H^0\circ hF\cong\nu^{-1}$ as functors $\cat{A}\to\cat{A}$.
                        
                        Proof: On $\Inj(\cat{A})$ we get
                        \begin{equation*}
                            (H^0\circ hF)^2\cong \nu^{-1}\circ H^0\circ hF\cong H^0\circ hF\circ\nu^{-1},
                        \end{equation*}
                        where we apply Step~1 using that $H^0\circ hF$ takes $\Inj(\cat{A})$ to $\Proj(\cat{A})$, and Step~3.
                        As both $\nu^{-1}$ and $H^0\circ hF$ take $\Inj(\cat{A})$ to $\Proj(\cat{A})$, and $H^0\circ hF$ is fully faithful on $\Proj(\cat{A})$ (and thus an equivalence to its image) by Step~4, it follows that $H^0\circ hF\cong\nu^{-1}$ on $\Inj(\cat{A})$.
                        Moreover, both functors are left exact, so the claim follows from this by replacing any object by an injective resolution and applying an argument similar to the proof of Step~1.
                    \item We have $hF\cong\mathrm{R}\nu^{-1}$ as functors $\D^+(\cat{A})\to\D^+(\cat{A})$.
        
                        Proof: Note that $\D^+_\infty(\cat{A})$ with the standard $t$-structure satisfies the assumptions of \cite[Thm.~1.3.3.2]{LurieHigherAlgebra} (in particular, it is right complete by \cite[Prop.~1.3.3.16]{LurieHigherAlgebra}).
                        Therefore $F\colon \D^+_\infty(\cat{A})\to\D^+_\infty(\cat{A})$ is up to isomorphism the only functor of $\infty$-categories restricting to $t_{\leq 0} \circ hF|_\cat{A}\in N(\Fun_{\mathrm{lex}}(\cat{A},\cat{A}))$, where $N(\Fun_{\mathrm{lex}}(\cat{A},\cat{A}))$ denotes the nerve (see \cite[\href{https://kerodon.net/tag/002M}{Tag~002M}]{Kerodon}) of the category of left exact functors $\cat{A}\to\cat{A}$.
                        On the other hand, by the above we know $H^0\circ hF|_\cat{A}\cong\nu^{-1}=t_{\leq 0}\circ h\mathbb{R}\nu^{-1}|_\cat{A}$ as ordinary functors $\cat{A}\to\cat{A}$, and thus they are also isomorphic in $N(\Fun_{\mathrm{lex}}(\cat{A},\cat{A}))$.
                        By \cite[Thm.~1.3.3.2]{LurieHigherAlgebra} it follows that $F\cong \mathbb{R}\nu^{-1}$ as functors of $\infty$-categories, and therefore $hF\cong h\mathbb{R}\nu^{-1}=\mathrm{R}\nu^{-1}$ as triangulated functors $\D^+(\cat{A})\to\D^+(\cat{A})$.
                    \item As $\cat{A}$ has finite global dimension, $\mathrm{R}\nu^{-1}\colon \Db(\cat{A})\to\Db(\cat{A})$ is the inverse Serre functor by \cref{fdalgserrefunctor}, and by Step~6 we have $hF\cong \mathrm{R}\nu^{-1}$.\qedhere
                \end{steps}
            \end{proof}
            Since $\ic{n}$ is a $\proj^n$-object in $\Dbc(\proj^n)$ by \cref{simples_P-obj}, we can consider the $\proj$-twist $\ptw{\ic{n}}\colon\Dbc(\proj^n)\to\Dbc(\proj^n)$ as in \cref{Ptwistdef}.
            By applying \cref{serrefunctorcharacterization} to $\ptw{\ic{n}}$, we obtain:
            \begin{theorem}\label{serrefunctortwist}
                The $\proj$-twist $\ptw{\ic{n}}\colon \Dbc(\proj^n)\to \Dbc(\proj^n)$ is the inverse Serre functor.
            \end{theorem}
            \begin{proof}
                In order to apply \cref{serrefunctorcharacterization} we first have to check the technical assumptions.

                Recall from \cref{setupDbcPn,SimpleStdProj} that $\Dbc(\proj^n)\cong\Db(\perv{\proj^n})$, and that the category $\perv{\proj^n}$ has finite global dimension, and enough projectives and enough injectives.
                By \cref{proj_inj}, all indecomposable projective objects except the projective cover $P_n$ of $\ic{n}$ are injective, and moreover there is an exact sequence $0\to P_n\to P_{n-1}\to P_{n-2}$ in $\perv{\proj^n}$ (this can be seen from the $\stan{}$-flags).
                
                Since $\ic{n}$ is a $\proj^n$-object in $\Db(\perv{\proj^n})\cong\Dbc(\proj^n)$, it is also $\proj^n$-like in $\D^+(\perv{\proj^n})$.
                We want to consider the $\proj$-twist $\ptw{\ic{n}}\colon\D^+(\perv{\proj^n})\to\D^+(\perv{\proj^n})$.
                To define this, we use the usual dg enhancement $\widetilde{\cat{D}}=\Ch^+(\Inj(\perv{\proj^n}))$ of $\D^+(\perv{\proj^n})$.
                By \cref{remarksontwists}, we can use $\RHom_{\perv{\proj^n}}(\ic{n},-)$ instead of $\Hom_{\widetilde{\cat{D}}}(\ic{n},-)$ to define the $\proj$-twist.
                Observe that $\RHom_{\perv{\proj^n}}(\ic{n},Y)$ is degreewise finite-dimensional for all $Y\in\Ch^+(\Inj(\perv{\proj^n}))$, and thus the tensor product $\RHom_\cat{D}(\ic{n},Y)\otimes\ic{n}$ exists in $\D^+(\perv{\proj^n})$ for all $Y\in\widetilde{\cat{D}}$, as required.
                Hence $\ptw{\ic{n}}\colon \D^+(\perv{\proj^n})\to\D^+(\perv{\proj^n})$ is well-defined.
                
                By construction, we have $\ptw{\ic{n}}=H^0(\dgptw{\ic{n}})\colon\D^+(\perv{\proj^n})\to \D^+(\perv{\proj^n})$, where $\dgptw{\ic{n}}\colon \widetilde{\cat{D}}\to\widetilde{\cat{D}}$ is a dg functor.
                By definition (see \cite[Variant~1.3.2.8]{LurieHigherAlgebra}) we have $\D^+_\infty(\perv{\proj^n})=N_{\mathrm{dg}}(\Ch^+(\Inj(\perv{\proj^n})))$.
                Here $N_{\mathrm{dg}}$ denotes the dg nerve from \cite[Constr.~1.3.1.6]{LurieHigherAlgebra}, see also \cite[\href{https://kerodon.net/tag/00PK}{Tag~00PK}]{Kerodon}.
                By \cite[Prop.~1.3.1.20]{LurieHigherAlgebra}, $N_{\mathrm{dg}}(\dgptw{\ic{n}})$ is a functor of $\infty$-categories $\D^+_\infty(\perv{\proj^n})\to\D^+_\infty(\perv{\proj^n})$, and by passing to homotopy categories, by \cite[Rem.~1.3.1.11]{LurieHigherAlgebra} we recover $hN_{\mathrm{dg}}(\dgptw{\ic{n}})=H^0(\dgptw{\ic{n}})=\ptw{\ic{n}}\colon \D^+(\perv{\proj^n})\to\D^+(\perv{\proj^n})$.
                
                It remains to check the conditions from \cref{serrefunctorcharacterization}:
                \begin{enumerate}
                    \item As $\D^+(\perv{\proj^n})^{\geq 0}$ is the extension closure of $\D^+(\perv{\proj^n})^{>0}$ and the IC sheaves, it suffices to show $\ptw{\ic{n}}(\D^+(\perv{\proj^n})^{>0})\subseteq\D^+(\perv{\proj^n})^{\geq 0}$ and $\ptw{\ic{n}}(\ic{k})\in \D^+(\perv{\proj^n})^{\geq 0}$ for all $0\leq k\leq n$.
                        
                        For this we use the triangles \cref{ptwistdiagram} defining the $\proj$-twist.
                        First, observe that $\ptw{\ic{n}}(\ic{n})\cong\ic{n}[-2n]\in \D^+(\perv{\proj^n})^{\geq 0}$.
                        Furthermore, for an object $X\in\D^+(\perv{\proj^n})^{>0}$ or $X=\ic{k}$ with $k<n$, $\RHom_{\perv{\proj^n}}(\ic{n},X)$ is cohomologically concentrated in positive degrees.
                        Thus $\RHom_{\perv{\proj^n}}(\ic{n},X)\otimes\ic{n}[-2]$ has cohomologies in degrees $>2$, and $\RHom_{\perv{\proj^n}}(\ic{n},X)\otimes\ic{n}$ has cohomologies in degrees $>0$, so $\cone{t^*\otimes\id-\id\otimes t}(X)$ has cohomologies in degrees $>0$.
                        As $X\in\D^+(\perv{\proj^n})^{\geq 0}$, it follows that $\ptw{\ic{n}}(X)\in\D^+(\perv{\proj^n})^{\geq 0}$.
                    \item Let $I\in\perv{\proj^n}$ be indecomposable injective.
                        If $I=I_k$ for $k<n$, then $\RHom_{\perv{\proj^n}}(\ic{n},I_k)=0$ and therefore $\ptw{\ic{n}}(I_k)\cong I_k=P_k$.

                        For $I=I_n=\costan{n}$ we have $\RHom_{\perv{\proj^n}}(\ic{n},I_n)\cong\kk$ (concentrated in degree~$0$), and so by evaluating \cref{ptwistdiagram} at $I_n$ we obtain the diagram
                        \begin{equation*}
                            \begin{tikzcd}
                                \ic{n}[-2]\arrow{r}{\epsilon^2_{n,n}}
                                &\ic{n}\arrow{r}\arrow{d}[swap]{\verdier(\mu_{n,n})}
                                &\cone{\epsilon^2_{n,n}}\arrow[dashed]{ld}\\
                                &I_n\arrow{ld}&\\
                                \ptw{\ic{n}}(I_n)&
                            \end{tikzcd}
                        \end{equation*}
                        where $\epsilon^2_{n,n}\colon \ic{n}[-2]\to\ic{n}$ is the generator of $\End_{\proj^n}^*(\ic{n})$.
                        
                        From the long exact sequence obtained by applying $\Hom_{\proj^n}(-,I_n)$ to the horizontal triangle and \cref{topIC_to_lower_nabla} it follows that $\Hom_{\proj^n}(\cone{\epsilon^2_{n,n}},I_n)$ is $1$-dimensional, i.e.~the induced morphism $\cone{\epsilon^2_{n,n}}\to I_n$ is unique up to scalar.
                        Therefore it suffices to find a non-split triangle of the form $\cone{\epsilon^2_{n,n}}\to I_n\to P_n\to \cone{\epsilon^2_{n,n}}[1]$. We know that $\Hom_{\proj^n}(I_n,P_n)=\Hom_{\proj^n}(\costan{n},\stan{n})$ is $1$-dimensional, spanned by the composition $f\colon I_n\longto{\verdier(\phi^0_{n-1,n})} \ic{n-1}\longto{\phi^0_{n-1,n}} P_n$.
                        From the octahedral axiom (using the triangle \cref{delta_comp_series} defining $P_n=\stan{n}$ as well as its dual) we obtain the diagram
                        \begin{equation*}
                            \begin{tikzcd}
                                I_n\arrow{r}{\verdier(\phi^0_{n-1,n})}\arrow{d}[swap]{=}
                                &\ic{n-1}\arrow{r}{\epsilon^1_{n-1,n}}\arrow{d}{\phi^0_{n-1,n}}
                                &\ic{n}[1]\arrow{r}\arrow{d}
                                &I_n[1]\arrow{d}{=}\\
                                I_n\arrow{r}{f}\arrow{d}[swap]{\verdier(\phi^0_{n-1,n})}
                                &P_n\arrow{r}\arrow{d}{=}
                                &\cone{f}\arrow{r}\arrow{d}
                                &I_n[1]\arrow{d}\\
                                \ic{n-1}\arrow{r}{\phi^0_{n-1,n}}
                                &P_n\arrow{r}
                                &\ic{n}\arrow{r}{\epsilon^1_{n,n-1}}\arrow{d}{\epsilon^2_{n,n}}
                                &\ic{n-1}[1]\arrow{d}{\epsilon^1_{n-1,n}}\\
                                &&\ic{n}[2]\arrow{r}{=}
                                &\ic{n}[2]\rlap{.}
                            \end{tikzcd}
                        \end{equation*}
                        Thus $\cone{f}[-1]\cong\cone{\epsilon^2_{n,n}\colon \ic{n}[-2]\to\ic{n}}$, and hence the second row is the desired triangle (up to rotation).
                    \item From the above it also follows that $\ptw{\ic{n}}$ is the identity functor on the full subcategory $\ProjInj(\perv{\proj^n})$.
                        Since the endomorphism algebra of the direct sum of the projective-injective objects is symmetric, we also have $\nu^{-1}\cong\id$ on $\ProjInj(\perv{\proj^n})$ by \cite[Prop.~3.5]{MR2369489}.\qedhere
                \end{enumerate}
            \end{proof}
            \begin{remark}
                In particular, \cref{serrefunctortwist} recovers the description of the Serre functor of $\Dbc(\proj^1)$ from \cite[\S3.1, p.~680]{MR2739061}, since $\ptw{\ic{1}}\cong\sptw{\ic{1}}^2$.
            \end{remark}
        \subsection{Other descriptions of the Serre functor}\label{otherserres}
            Recall the equivalences $\Dbc(\proj^n)\cong\Dbc(\catO^\lie{p}_0(\lie{sl}_{n+1}(\kk)))\cong\Db(A_n\lmodfd)$ mentioned in \cref{otherperspectives}.
            The Serre functor of $\Dbc(\proj^n)$ also has explicit descriptions in terms of finite-dimensional algebras and in terms of Lie algebras, and furthermore there is a description of the Serre functor for the constructible category of the full flag variety.
            We summarize these results and explain how they are related to \cref{serrefunctortwist}.
            
            In terms of finite-dimensional algebras, the Serre functor is given by the derived functor of the Nakayama functor by results of Happel \cite[Prop.~4.10]{Happel}.
            This actually underlies our argument, as the proof of the criterion \cref{serrefunctorcharacterization} (which we adapted from \cite[Thm.~3.4]{MR2369489}) compares the candidate Serre functor with the derived functor of the Nakayama functor.

            In \cite{MR2369489} Mazorchuk and Stroppel provide a description of the Serre functor of $\Db(\catO^\lie{p}_0(\lie{sl}_{n+1}(\kk)))$ in Lie-theoretic language.
            For this, the criterion \cite[Thm.~3.4]{MR2369489} is first used to show that the Serre functor of $\Db(\catO_0(\lie{sl}_{n+1}(\kk)))$ is given by the (derived) shuffling functor $\Sh_{w_0}^2$, where $\Sh_{w_0}=\Sh_{s_{i_1}}\dots\Sh_{s_{i_r}}$ for a reduced expression $s_{i_1}\dots s_{i_r}$ of the longest element $w_0$ of the Weyl group.
            Alternatively, the Serre functor of $\Db(\catO_0(\lie{sl}_{n+1}(\kk)))$ is also isomorphic to the (derived) Arkhipov twisting functor $\Tw_{w_0}^2$, where $\Tw_{w_0}=\Tw_{s_{i_1}}\dots\Tw_{s_{i_r}}$.
            
            By \cite[Prop.~4.4]{MR2369489}, the Serre functor of $\Db(\catO^\lie{p}_0(\lie{sl}_{n+1}(\kk)))$ is then $\Sh_{w_0}^2[-2\ell(w_0^\lie{p})]$, where $w_0^\lie{p}\in W_\lie{p}\cong S_1\times S_n$ is the longest element of the parabolic Weyl group.
            Rather than applying the criterion \cref{serrefunctorcharacterization}, the proof uses the inclusion $\Db(\catO^\lie{p}_0(\lie{sl}_{n+1}(\kk)))\hookto\Db(\catO_0(\lie{sl}_{n+1}(\kk)))$ and its left and right adjoints (i.e.~the (derived) Zuckerman functors) to ``push down'' the description of the Serre functor from $\Db(\catO_0(\lie{sl}_{n+1}(\kk)))$.
            In particular, note that since the inclusion $\Db(\catO^\lie{p}_0(\lie{sl}_{n+1}(\kk)))\hookto\Db(\catO_0(\lie{sl}_{n+1}(\kk)))$ is not full, the Serre functor of $\Db(\catO^\lie{p}_0(\lie{sl}_{n+1}(\kk)))$ is not the restriction of that of $\Db(\catO_0(\lie{sl}_{n+1}(\kk)))$.

            In the language of perverse sheaves, in \cite{MR2119139} Beilinson, Bezrukavnikov and Mir\-ko\-vi\'c provide a description of the Serre functor for the full flag variety $G/B$, where (for us) $G=\GL_{n+1}(\kk)$ and $B\subseteq G$ is the usual Borel subgroup of upper triangular matrices.
            By \cite[Prop.~2.5]{MR2119139}, the Serre functor of $\Dbc(G/B)$ is given by the Radon transform $(R_{w_0}^*)^2$, where $R_{w_0}^*=R_{s_{i_1}}^*\dots R_{s_{i_r}}^*$ .
            
            Under the equivalences $\catO_0(\lie{sl}_{n+1}(\kk))\cong\perv{G/B}$ and $\catO^\lie{p}_0(\lie{sl}_{n+1}(\kk))\cong\perv{G/P}=\perv{\proj^n}$, the inclusion functor corresponds to $\pi^![-d]\cong\pi^*[d]$ and the (dual) Zuckerman functors are $\pi_![d]$ and $\pi_*[-d]$, see e.g.~\cite[p.~504, Rem.~(2)]{MR1322847}.
            Here $P\subseteq G$ is the parabolic subgroup with block sizes $(n,1)$, $\pi\colon G/B\to G/P$ is the canonical map, and $d=\dim G/B-\dim G/P=\ell(w_0^\lie{p})$.
            Using this, one can apply the purely formal argument from \cite[Prop.~4.4]{MR2369489} to obtain a description of the Serre functor of $\Dbc(\proj^n)$ from the description of the Serre functor of $\Dbc(G/B)$.
            
            Combining \cref{serrefunctortwist} with the above observations yields the following relation between $\ptw{\ic{n}}$ and the Radon transform $R_{w_0}^!$, and also a decomposition of $\ptw{\ic{n}}$ into a sequence of spherical twists:
            \begin{corollary}\label{someconsequences}\leavevmode
                \begin{enumerate}
                    \item The square
                        \begin{equation*}
                            \begin{tikzcd}[column sep=large]
                                \Dbc(G/B)\arrow{r}{(R_{w_0}^!)^2[2d]}
                                &\Dbc(G/B)\arrow{d}{\pi_*[-d]}\\
                                \Dbc(\proj^n)\arrow{u}{\pi^*[d]}\arrow{r}{\ptw{\ic{n}}}
                                &\Dbc(\proj^n)
                            \end{tikzcd}
                        \end{equation*}
                        commutes up to natural isomorphism.
                    \item For any reduced expression $w_0=s_{i_1}\dots s_{i_r}$ there is a natural isomorphism
                        \begin{equation*}
                            \ptw{\ic{n}}[2\ell(w_0)-2\ell(w_0^\lie{p})]\cong (\sptw{P_{n-i_1}}\dots\sptw{P_{n-i_r}})^2.
                        \end{equation*}
                \end{enumerate}
            \end{corollary}
            \begin{proof}\leavevmode
                \begin{enumerate}
                    \item By \cite[Prop.~4.1 and Prop.~4.4]{MR2369489}, the inverse Serre functor of $\Db(\catO^\lie{p}_0(\lie{sl}_{n+1}(\kk)))$ is $\hat{Z}\Serre_\lie{b}^{-1} \incl [2d]$, where $\incl\colon \Db(\catO^\lie{p}_0(\lie{sl}_{n+1}(\kk)))\hookto\Db(\catO_0(\lie{sl}_{n+1}(\kk)))$, $\Serre_\lie{b}^{-1}$ is the inverse Serre functor of $\Db(\catO_0(\lie{sl}_{n+1}(\kk)))$, and $\hat{Z}$ the dual Zuckerman functor.
                        In geometric language, $\incl$ is $\pi^*[d]$ and $\hat{Z}$ is $\pi_*[-d]$, and $(R_{w_0}^!)^2$ is the inverse Serre functor of $\Dbc(G/B)$ by \cite[Prop.~2.5 and Fact~2.2]{MR2119139}.
                        The claim then follows from \cref{serrefunctortwist} and uniqueness of the Serre functor.
                    \item By \cref{serrefunctortwist} and \cite[Prop.~4.4]{MR2369489} the inverse Serre functor of $\Dbc(\proj^n)\cong\Db(\catO^\lie{p}_0(\lie{sl}_{n+1}(\kk)))$ is
                        \begin{equation*}
                            \ptw{\ic{n}}\cong\Serre^{-1}\cong\Sh_{w_0}^{-2}[2\ell(w_0^\lie{p})]=(\Sh_{s_{i_r}}^{-1}\dots \Sh_{s_{i_1}}^{-1})^2[2\ell(w_0^\lie{p})].
                        \end{equation*}
                        By \cite[Thm.~4.14]{MR4239691} there is a natural isomorphism $\Sh_{s_i}^{-1}\cong\sptw{P_{n-i}}[-1]$, which proves the claim.\qedhere
                \end{enumerate}
            \end{proof}
            For $n=1$, we in particular get $\sptw{\ic{1}}^2\cong\ptw{\ic{1}}\cong\sptw{P_0}^2[-2]$, and in fact we even have $\sptw{\ic{1}}\cong\sptw{P_0}[-1]$ by \cite[Thm.~3.6 and Thm.~3.10]{MR4239691}.
    \section{Classification of \texorpdfstring{$\PP$}{P}-objects in \texorpdfstring{$\perv{\PP^n}$}{Perv(Pn)}}\label{PObjectClassification}
        In this section we classify the $\PP$-objects and $\PP$-like objects in $\perv{\PP^n}\subset \Dbc(\PP^n)$.
        For this we first determine the indecomposable Calabi--Yau objects.
        After that, we introduce certain string objects, and show that all of them are $\proj$-like.
        \subsection{Calabi--Yau objects}\label{CYclassification}
            The following easy lemma provides obstructions for Calabi--Yau objects in the presence of projective-injective objects.
            \begin{lemma}\label{projinjCY}
                Let $\cat{A}$ be a finite-length Hom-finite Krull--Schmidt abelian category with enough projectives and finite global dimension, and let $P\in\cat{A}$ be an indecomposable projective-injective object.
                \begin{enumerate}
                    \item If $P$ has isomorphic top and socle, then $P$ is $0$-Calabi--Yau in $\Db(\cat{A})$. 
                    \item If $X\in\cat{A}$ involves a composition factor $\mathrm{top}(P)$ or $\mathrm{soc}(P)$, then $X$ cannot be $d$-Calabi--Yau for $d>0$.
                \end{enumerate}
            \end{lemma}
            \begin{proof}\leavevmode
                \begin{enumerate}
                    \item This is clear since the Serre functor, which is the Nakayama functor, fixes such projective-injective objects.
                    \item Let $P\in\cat{A}$ be projective-injective and assume that $X\in\cat{A}$ is $d$-Calabi--Yau with $d>0$.
                        By definition, this means that the composition pairing
                        \begin{equation*}
                            \Hom_{\Db(\cat{A})}(P,X[r])\otimes\Hom_{\Db(\cat{A})}(X,P[d-r])\to\Hom_{\Db(\cat{A})}(X,X[d])
                        \end{equation*}
                        is non-degenerate for all $r\in\Z$.
                        If $X$ involves a simple subquotient $\mathrm{top}(P)$, then for $r=0$ the first tensor factor is non-zero while the second one is not.
                        If $X$ involves a simple subquotient $\mathrm{soc}(P)$, then for $r=d$ the second tensor factor is non-zero while the first is not.
                        Thus in these cases the pairing cannot be non-degenerate, a contradiction.\qedhere
                \end{enumerate}
            \end{proof}
            As an application, we recover the classification of the indecomposable Calabi--Yau objects in $\perv{\PP^n}$, which was obtained algebraically in \cite[\S 7.4]{MazorchukVII}.
            \begin{corollary}\label{CY_objects}
                An indecomposable object $E\in\perv{\PP^n}$ is Calabi--Yau if and only if $E\in\{\ic{n}\}\cup\{P_i\mid 0\leq i\leq n-1\}$.
            \end{corollary}
            \begin{proof}
                The simple object $\ic{n}$ is $2n$-Calabi--Yau by \cref{simples_P-obj}, while the projective-injective objects $P_i\in\perv{\PP^n}$ for $0\leq i\leq n-1$ are $0$-Calabi--Yau by \cref{projinjCY}.

                It is obvious that objects in the heart of a $t$-structure cannot be $d$-Calabi--Yau for $d<0$, and that only projective objects in the heart can be $0$-Calabi--Yau: if $E\in\perv{\proj^n}$ is $0$-Calabi--Yau, then $\Hom_{\proj^n}(E,X[1])\cong\Hom_{\proj^n}(X[1],E)^\vee=0$ for all $X\in\perv{\proj^n}$, so $E$ has to be projective.
                Moreover, by \cref{projinjCY} no object in $\perv{\proj^n}$ involving simple subquotients $\ic{k}$ for $0\leq k<n$ can be $d$-Calabi--Yau with $d>0$.
                As $\Hom_{\proj^n}(\ic{n},\ic{n}[1])=0$, the only indecomposable object such that all of its simple subquotients are $\ic{n}$ is $\ic{n}$ itself.
            \end{proof}
            Alternatively, in the proof of \cref{CY_objects} one can also use the classification of indecomposable perverse sheaves (using the language of finite-dimensional algebras), and \cite[Prop.~2]{MR4090927} or \cref{zigzaghoms} below, to show that there are no $0$-Calabi--Yau objects besides the projective-injective objects.
        \subsection{String objects}\label{strings}
            For $0\leq b\leq a\leq n$ we recursively define the \emph{string objects} $\zz{+}{a}{b}$ as follows.
            Set $\zz{+}{a}{a}=\ic{a}$ and $\zz{+}{a}{a-1}=\stan{a}$, and for $a\geq b+2$ define $\zz{+}{a}{b}=\cone{\psi_{a-2,b}}[-1]$, where $\psi_{a-2,b}\colon\zz{+}{a-2}{b}\to\stan{a}[1]$ is a non-zero morphism that will be fixed recursively.
            Hence $\zz{+}{a}{b}$ fits into a triangle
            \begin{equation}
                \label{delta flag of zigzags}
                \stan{a}\longto{\iota_{a,b}}\zz{+}{a}{b}\longto{\pi_{a,b}}\zz{+}{a-2}{b}\longto{\psi_{a-2,b}}\stan{a}[1].
            \end{equation}
            We also define $\zz{-}{a}{b}=\verdier(\zz{+}{a}{b})$; alternatively these can be obtained inductively by the dual construction.
            Note that the string objects $\zz{\pm}{a}{b}$ lie in $\perv{\proj^n}$.
            
            To properly define $\psi_{a-2,b}$, so that $\zz{+}{a}{b}$ is well-defined, we need:
            \begin{lemma}\label{stringwelldef}
                Let $a\geq b+2$ and assume by induction that $\zz{+}{a-2}{b}$ is already defined.
                Then $\Hom_{\proj^n}(\zz{+}{a-2}{b},\stan{a}[1])$ is $1$-dimensional.
            \end{lemma}
            \begin{proof}
                For $a=b+2$ and $a=b+3$, the claim follows from \cref{mor_simple_stand} respectively \cref{DeltaHoms}.

                For $a>b+3$ we apply $\Hom_{\proj^n}(-,\stan{a}[1])$ to the triangle \cref{delta flag of zigzags}, which by induction is unique up to rescaling of the morphisms.
                From the construction of $\zz{+}{a-4}{b}$ and \cref{DeltaHoms,mor_simple_stand} it follows that $\Hom_{\proj^n}(\zz{+}{a-4}{b},\stan{a}[r])=0$ for $r\leq 2$, and thus
                \begin{equation*}
                    \Hom_{\proj^n}(\zz{+}{a-2}{b},\stan{a}[1])\cong\Hom_{\proj^n}(\stan{a-2},\stan{a}[1]).
                \end{equation*}
                By \cref{DeltaHoms}, this is $1$-dimensional, as claimed.
            \end{proof}
            From the proof we obtain the following explicit definition of a canonical non-zero morphism $\psi_{a-2,b}\colon \zz{+}{a-2}{b}\to\stan{a}[1]$:
            \begin{itemize}
                \item For $a=b+2$, we take $\psi_{a-2,a-2}=\phi^1_{a-2,a}\colon\ic{a-2}\to\stan{a}[1]$.
                \item For $a=b+3$, we define the morphism $\psi_{a-2,a-3}$ as the composition $\stan{a-2}\longto{\mu_{a-2,a-1}}\ic{a-1}[1]\longto{\phi^0_{a-1,a}}\stan{a}[1]$.
                \item For $a>b+3$, the morphism $\psi_{a-2,b}\colon\zz{+}{a-2}{b}\to\stan{a}[1]$ is uniquely defined by the commutative diagram
                    \begin{equation*}
                        \begin{tikzcd}[column sep=width("aaaaaaa")]
                            \zz{+}{a-2}{b}\arrow[dashed]{r}{\psi_{a-2,b}}
                            &\stan{a}[1]\\
                            \stan{a-2}\rlap{,}\arrow{u}{\iota_{a-2,b}}\arrow{ru}[swap]{\delta_{a-2,a}^1}&
                        \end{tikzcd}
                    \end{equation*}
                    where $\iota_{a-2,b}$ is fixed by the choice of the triangle \cref{delta flag of zigzags} defining $\zz{+}{a-2}{b}$.
            \end{itemize}
            This completes the construction of the string objects $\zz{\pm}{a}{b}$.
            
            In terms of finite-dimensional algebras, the string objects as defined above are precisely the \emph{string modules} mentioned in \cref{otherperspectives}.
        \subsection{String objects are \texorpdfstring{$\PP$}{P}-like}\label{stringmorphisms}
            We show that all the string objects $\zz{\pm}{a}{b}$ are $\proj^k$-like, where $k$ depends on $a$ and $b$ by an explicit formula.
            Note that with the exception of $\zz{\pm}{n}{n}=\ic{n}$, the string objects cannot be $\proj^k$-objects, since \cref{CY_objects} obstructs them from having the Calabi--Yau property.
            \subsubsection{Morphisms between string objects and IC sheaves}\label{prel_results}
                We want to understand the total endomorphism spaces $\End_{\proj^n}^*(\zz{\pm}{a}{b})$.
                Due to the inductive definition of the string objects, we first need to compute morphisms between $\zz{\pm}{a}{b}$ and some IC sheaves.
                \begin{lemma}\label{HomZZplusIC}
                    Let $0\leq b\leq a\leq n$ and $r\in\Z$.
                    \begin{enumerate}
                        \item If $a-b$ is even, then
                            \begin{equation*}
                                \Hom_{\proj^n}(\zz{+}{a}{b},\ic{b}[r])\cong\Hom_{\proj^n}(\ic{b},\zz{-}{a}{b}[r])\cong\begin{cases}
                                    \kk&\text{if $0\leq r\leq 2b$ and $r$ even},\\
                                    0&\text{otherwise}.
                                \end{cases}
                            \end{equation*}
                        \item If $a-b$ is odd, then
                            \begin{equation*}
                                \Hom_{\proj^n}(\zz{+}{a}{b},\ic{b}[r])\cong\Hom_{\proj^n}(\ic{b},\zz{-}{a}{b}[r])\cong 0.
                            \end{equation*}
                    \end{enumerate}
                \end{lemma}
                \begin{proof}
                    In both cases the first isomorphism is given by Verdier duality, so it suffices to compute $\Hom_{\proj^n}(\zz{+}{a}{b},\ic{b}[r])$.
    
                    We prove both statements by induction on $a-b$.
                    As the argument for the inductive step is the same in both cases, we only give details for the first statement.
                    \begin{enumerate}
                        \item For the base case $a=b$, we have $\zz{+}{a}{b}=\ic{b}$ and the claim is immediate by \cref{maps_between_simples}.
                    
                            For the inductive step, we apply the functor $\Hom_{\proj^n}(-,\ic{b})$ to \cref{delta flag of zigzags} to get the long exact sequence
                            \[
                                \ldots \to \Hom_{\proj^n}(\zz{+}{a-2}{b},\ic{b}[r]) \to \Hom_{\proj^n}(\zz{+}{a}{b},\ic{b}[r]) \to \Hom_{\proj^n}(\stan{a},\ic{b}[r]) \to \ldots
                            \]
                            As $a>b$, we have $\Hom_{\proj^n}(\stan{a},\ic{b}[r])=0$ for all $r$ by \cref{topIC_to_lower_nabla}.
                            Hence
                            \begin{equation*}
                                \Hom_{\proj^n}(\zz{+}{a}{b},\ic{b}[r])\cong\Hom_{\proj^n}(\zz{+}{a-2}{b},\ic{b}[r]),
                            \end{equation*}
                            and so the statement follows from the inductive hypothesis.
                        \item For the base case $a=b+1$, we have $\zz{+}{a}{b}=\stan{a}$ and $\Hom_{\proj^n}(\stan{a},\ic{b}[r])=0$ for any $r$ by \cref{topIC_to_lower_nabla}.
                        The claim follows from this by the same arguments as in the first part.\qedhere
                    \end{enumerate}
                \end{proof}
                The proof for $a-b$ even also yields a canonical non-zero morphism $m_{a,b}^r\colon\zz{+}{a}{b}\to\ic{b}[r]$, which is defined as the composition
                \begin{equation*}
                    \zz{+}{a}{b}\longto{\pi_{a,b}}\zz{+}{a-2}{b}\longto{\pi_{a-2,b}}\dots\longto{\pi_{b+2,b}}\ic{b}\longto{\epsilon^r_{b,b}}\ic{b}[r].
                \end{equation*}
                Since $\epsilon^{r+2}_{b,b}=\epsilon^2_{b,b}\epsilon^r_{b,b}$, we also obtain the commutative diagram
                \begin{equation}
                    \label{HomZZplusICmodule}
                    \begin{tikzcd}
                        \zz{+}{a}{b}\arrow{d}[swap]{m_{a,b}^r}\arrow{rd}{m_{a,b}^{r+2}}&\\
                        \ic{b}[r]\arrow{r}{\epsilon^2_{b,b}}
                        &\ic{b}[r+2]\rlap{.}
                    \end{tikzcd}
                \end{equation}
                \begin{lemma}\label{HomZZtopIC}
                    Let $0\leq b\leq a\leq n$ and $r\in\Z$.
                    \begin{enumerate}
                        \item If $a-b$ is even, then
                            \begin{equation*}
                                \Hom_{\proj^n}(\zz{+}{a}{b},\ic{a}[r])\cong\Hom_{\proj^n}(\ic{a},\zz{-}{a}{b}[r])\cong\begin{cases}
                                    \kk&\text{if $0\leq r\leq a+b$ and $r$ even},\\
                                    0&\text{otherwise}.
                                \end{cases}
                            \end{equation*}
                        \item If $a-b$ is odd, then
                            \begin{equation*}
                                \Hom_{\proj^n}(\zz{+}{a}{b},\ic{a}[r])\cong\Hom_{\proj^n}(\ic{a},\zz{-}{a}{b}[r])\cong\begin{cases}
                                    \kk&\text{if $0\leq r\leq a-b-1$ and $r$ even},\\
                                    0&\text{otherwise}.
                                \end{cases}
                            \end{equation*}
                        \item For $0\leq i\leq n-a$ we have
                            \begin{equation*}
                                \Hom_{\proj^n}(\zz{+}{a}{b},\ic{a+i}[r])\cong\Hom_{\proj^n}(\zz{+}{a}{b},\ic{a}[r-i]).
                            \end{equation*}
                    \end{enumerate}
                \end{lemma}
                \begin{proof}
                    For (1) and (2), by Verdier duality it suffices to compute $\Hom_{\proj^n}(\zz{+}{a}{b},\ic{a}[r])$.
                    We prove the claim by induction on $a-b$.
                    The base cases are $a=b$ for $a-b$ even, respectively $a=b+1$ for $a-b$ odd, and for these the claim holds by \cref{maps_between_simples} and \cref{mor_simple_stand}, respectively.
                
                    For the inductive step, we apply $\Hom_{\proj^n}(-,\ic{a})$ to \cref{delta flag of zigzags} to get the long exact sequence
                    \begin{equation*}
                        \dots\to\Hom_{\proj^n}(\zz{+}{a-2}{b},\ic{a}[r])\to\Hom_{\proj^n}(\zz{+}{a}{b},\ic{a}[r])\to\Hom_{\proj^n}(\stan{a},\ic{a}[r])\to\dots.
                    \end{equation*}
                    By \cref{mor_simple_stand}, $\Hom_{\proj^n}(\stan{a},\ic{a}[r])$ is $1$-dimensional for $r=0$ and vanishes otherwise.
                    By restriction to $\proj^{a-2}$, we have $\Hom_{\proj^n}(\zz{+}{a-2}{b},\ic{a}[r])\cong\Hom_{\proj^n}(\zz{+}{a-2}{b},\ic{a-2}[r-2])$.
                    In particular, this is concentrated in degrees $\geq 2$, and hence the claim follows.
                    
                    The last claim is immediate from $\imath_a^!\ic{a+i}\cong\ic{a}[-i]$.
                \end{proof}
                \begin{remark}\label{HomZZtopICdiag}
                    From the proof of \cref{HomZZtopIC} we obtain the following explicit description of a morphism $n^r_{a,b}$ spanning $\Hom_{\proj^n}(\zz{+}{a}{b},\ic{a}[r])$.
                    For $r=0$,  we define $n^0_{a,b}\colon \zz{+}{a}{b}\to\ic{a}$ as the unique morphism making the diagram
                    \begin{equation*}
                        \begin{tikzcd}[column sep=width("aaaaaaa")]
                            \zz{+}{a-r}{b}\arrow[dashed]{r}{n^0_{a-r,b}}
                            &\ic{a-r}\\
                            \stan{a-r}\arrow{u}{\iota_{a-r,b}}\arrow{ru}[swap]{\mu_{a-r,a-r}}&
                        \end{tikzcd}
                    \end{equation*}
                    commute.
                    For $0<r\leq a-b$, we define $n^r_{a,b}\colon \zz{+}{a}{b}\to\ic{a}[r]$ as the composition
                    \begin{equation*}
                        \zz{+}{a}{b}\longto{\pi_{a,b}}\zz{+}{a-2}{b}\longto{\pi_{a-2,b}}\dots\longto{\pi_{a-r+2,b}}\zz{+}{a-r}{b}\longto{n^0_{a-r,b}}\ic{a-r}\longto{\epsilon^r_{a-r,a}}\ic{a}[r].
                    \end{equation*}
                    For $r>a-b$ (this can only happen if $a-b$ is even), we define $n^r_{a,b}$ as the composition
                    \begin{equation*}
                        \zz{+}{a}{b}\longto{\pi_{a,b}}\zz{+}{a-2}{b}\longto{\pi_{a-2,b}}\dots\longto{\pi_{b+2,b}}\ic{b}\longto{\epsilon^r_{b,a}}\ic{a}[r].
                    \end{equation*}
                    Moreover, for $0\leq i\leq n-a$, a canonical morphism $\zz{+}{a}{b}\to\ic{a+i}[r]$ is given by the composition
                    \begin{equation*}
                        \zz{+}{a}{b}\longto{n^{r-i}_{a,b}}\ic{a}[r-i]\longto{\epsilon^i_{a,a+i}}\ic{a+i}[r].
                    \end{equation*}
                \end{remark}
                \begin{remark}\label{2SESsrem}
                    The proof of \cref{HomZZtopIC} and the octahedral axiom for the composition $n^0_{a,b}\iota_{a,b}=\mu_{a,a}$ also yields triangles
                    \begin{equation}\label{2SESs}
                        \zz{-}{a-1}{b}\to \zz{+}{a}{b}\longto{n^0_{a,b}}\ic{a}\to\zz{-}{a-1}{b}[1].
                    \end{equation}
                    By Verdier duality one also obtains triangles $\ic{a}\to\zz{-}{a}{b}\to\zz{+}{a-1}{b}\to\ic{a}[1]$.
                \end{remark}
                \begin{lemma}\label{HomZZtopICmodule}
                    Let $0\leq b\leq a\leq n$ and $0\leq r\leq a+b-2$ if $a-b$ is even, respectively $0\leq r\leq a-b-3$ if $a-b$ is odd.
                    Then the diagram
                    \begin{equation*}
                        \begin{tikzcd}
                            \zz{+}{a}{b}\arrow{r}{n_{a,b}^r}\arrow{rd}[swap]{n_{a,b}^{r+2}}
                            &\ic{a}[r]\arrow{d}{\epsilon^2_{a,a}}\\
                            &\ic{a}[r+2]
                        \end{tikzcd}
                    \end{equation*}
                    commutes up to a non-zero scalar.
                \end{lemma}
                \begin{proof}
                    We prove the claim by induction on $a-b$.
                    In the base cases $\End_{\proj^n}^*(\ic{a})$ respectively $\Hom_{\proj^n}^*(\stan{a},\ic{a})$ (depending on the parity of $a-b$) there is nothing to show.
                    
                    For the inductive step, for $r>0$ \cref{HomZZtopIC} and its proof yields the commutative diagram
                    \begin{equation*}
                        \begin{tikzcd}[column sep=small]
                            \Hom_{\proj^n}(\zz{+}{a}{b},\ic{a}[r])\arrow{r}{\cong}\arrow{d}
                            &\Hom_{\proj^n}(\zz{+}{a-2}{b},\ic{a}[r])\arrow{r}{\cong}\arrow{d}
                            &\Hom_{\proj^n}(\zz{+}{a-2}{b},\ic{a-2}[r-2])\arrow{d}\\
                            \Hom_{\proj^n}(\zz{+}{a}{b},\ic{a}[r+2])\arrow{r}{\cong}
                            &\Hom_{\proj^n}(\zz{+}{a-2}{b},\ic{a}[r+2])\arrow{r}{\cong}
                            &\Hom_{\proj^n}(\zz{+}{a-2}{b},\ic{a-2}[r])\rlap{.}
                        \end{tikzcd}
                    \end{equation*}
                    Here the vertical morphisms are given by postcomposition with $\epsilon^2_{a,a}$ and $\imath_{a-2}^!(\epsilon^2_{a,a})$, respectively.
                    By \cref{ICdegree2end} we have $\imath_{a-2}^!(\epsilon^2_{a,a})=\epsilon^2_{a-2,a-2}$, and therefore by induction it only remains to show the claim for $r=0$.
                    
                    For $r=0$, the square
                    \begin{equation*}
                        \begin{tikzcd}
                            \zz{+}{a}{b}\arrow{r}{\pi_{a,b}}\arrow{d}[swap]{n^0_{a,b}}
                            &\zz{+}{a-2}{b}\arrow{d}{n^2_{a-2,b}}\\
                            \ic{a}\arrow{r}{\epsilon^2_{a,a}}
                            &\ic{a}[2]
                        \end{tikzcd}
                    \end{equation*}
                    commutes up to a possibly zero scalar, as $\cone{\pi_{a,b}}=\stan{a}[1]$ and $\Hom_{\proj^n}(\stan{a},\ic{a}[r])=0$ for $r>0$.
                    By \cref{HomZZtopICdiag}, the composition $n^2_{a,b}=n^2_{a-2,b}\pi_{a,b}$ is non-zero, and thus we have to show that it factors through $\epsilon^2_{a,a}\colon \ic{a}\to\ic{a}[2]$.
                    
                    For this we apply $\Hom_{\proj^n}(-,\ic{a})$ to the triangle \cref{2SESs} to get the long exact sequence
                    \begin{equation*}
                        \dots\to\Hom_{\proj^n}(\ic{a},\ic{a}[2])\to \Hom_{\proj^n}(\zz{+}{a}{b},\ic{a}[2])\to \Hom_{\proj^n}(\zz{-}{a-1}{b},\ic{a}[2])\to\dots.
                    \end{equation*}
                    By \cref{maps_between_simples,HomZZtopIC} we know that $\Hom_{\proj^n}(\ic{a},\ic{a}[2])$ is spanned by the generator $\epsilon^2_{a,a}$ of $\End_{\proj^n}^*(\ic{a})$ and $\Hom_{\proj^n}(\zz{+}{a}{b},\ic{a}[2])$ is spanned by the composition $n^2_{a-2,b}\pi_{a,b}$, so it suffices to show $\Hom_{\proj^n}(\zz{-}{a-1}{b},\ic{a}[2])=0$.
                    
                    By restriction to $\proj^{a-1}$, we have $\Hom_{\proj^n}(\zz{-}{a-1}{b},\ic{a}[2])\cong\Hom_{\proj^n}(\zz{-}{a-1}{b},\ic{a-1}[1])$, and that this vanishes can be seen by applying $\Hom_{\proj^n}(-,\ic{a-1})$ to the Verdier-dual version of the triangle \cref{2SESs} defining $\zz{-}{a-1}{b}$.
                \end{proof}
                \begin{lemma}\label{HomZZminusIC}
                    Let $0\leq b\leq a\leq n$ and $r\in\Z$.
                    \begin{enumerate}
                        \item If $a-b$ is even, then
                            \begin{equation*}
                                \Hom_{\proj^n}(\ic{b},\zz{+}{a}{b}[r])\cong\Hom_{\proj^n}(\zz{-}{a}{b},\ic{b}[r])\cong\begin{cases}
                                    \kk&\text{if $a-b\leq r\leq a+b$ and $r$ even},\\
                                    0&\text{otherwise}.
                                \end{cases}
                            \end{equation*}
                        \item If $a-b$ is odd, then
                            \begin{equation*}
                                \Hom_{\proj^n}(\ic{b},\zz{+}{a}{b}[r])\cong\Hom_{\proj^n}(\zz{-}{a}{b},\ic{b}[r])\cong\begin{cases}
                                    \kk&\text{if $0\leq r\leq\min(2b,a-b)$ and $r$ even},\\
                                    \kk&\text{if $\max(2b,a-b)\leq r\leq a+b$ and $r$ odd},\\
                                    0&\text{otherwise}.
                                \end{cases}
                            \end{equation*}
                    \end{enumerate}
                \end{lemma}
                \begin{proof}
                    By Verdier duality it suffices to compute $\Hom_{\proj^n}(\ic{b},\zz{+}{a}{b}[r])$.
                    \begin{enumerate}
                        \item Applying the functor $\Hom_{\proj^n}(\ic{b},-)$ to the triangle \cref{2SESs} yields the long exact sequence
                            \[
                                \ldots \to \Hom_{\proj^n}(\ic{b}, \zz{-}{a-1}{b}[r]) \to \Hom_{\proj^n}(\ic{b},\zz{+}{a}{b}[r]) \to \Hom_{\proj^n}(\ic{b},\ic{a}[r]) \to \dots.
                            \]
                            As $a-b$ is even, we have $\Hom_{\proj^n}(\ic{b},\zz{-}{a-1}{b}[r])\cong 0$ by \cref{HomZZplusIC}.(2).
                            Therefore $\Hom_{\proj^n}(\ic{b},\zz{+}{a}{b}[r]) \cong \Hom_{\proj^n}(\ic{b},\ic{a}[r])$ and the claim follows by \cref{maps_between_simples}.
                         \item We distinguish two cases to avoid having to determine connecting morphisms (note that $a=3b$ is impossible since $a-b$ is odd).
                         
                            If $a>3b$, we apply the functor $\Hom_{\proj^n}(\ic{b},-)$ to the triangle \cref{2SESs} to get the long exact sequence
                            \[
                                \ldots \to \Hom_{\proj^n}(\ic{b},\zz{-}{a-1}{b}[r]) \to \Hom_{\proj^n}(\ic{b},\zz{+}{a}{b}[r]) \to \Hom_{\proj^n}(\ic{b},\ic{a}[r]) \to \dots,
                            \]   
                            and the claim follows by using \cref{HomZZplusIC}.(1) for the left-hand side and \cref{maps_between_simples} for the right-hand side.  
    
                            If $a<3b$, we do induction on $a-b$.
                            The base case $a=b+1$ is given by \cref{mor_simple_stand}.
                            For the inductive step, applying the functor $\Hom_{\proj^n}(\ic{b},-)$ to  \cref{delta flag of zigzags} yields the long exact sequence
                                \[
                                    \ldots \to \Hom_{\proj^n}(\ic{b},\stan{a}[r]) \to \Hom_{\proj^n}(\ic{b},\zz{+}{a}{b}[r]) \to \Hom_{\proj^n}(\ic{b},\zz{+}{a-2}{b}[r]) \to \dots.
                                \]
                            The claim follows by using \cref{mor_simple_stand} for the left-hand side and the inductive hypothesis for the right-hand side.\qedhere
                    \end{enumerate}
                \end{proof}
                \begin{remark}
                    For $a-b$ odd, the case distinction in the proof of \cref{HomZZminusIC} avoids analysis of the connecting morphisms.
                    Note that to compute $\Hom_{\proj^n}(\ic{b},\zz{+}{a}{b})$ (i.e.~the case $r=0$) one can always use the argument for $a<3b$ to obtain
                    \begin{equation*}
                        \Hom_{\proj^n}(\ic{b},\zz{+}{a}{b})\cong\Hom_{\proj^n}(\ic{b},\zz{+}{a-2}{b})\cong\dots\cong\Hom_{\proj^n}(\ic{b},\stan{b+1}).
                    \end{equation*}
                    Hence a canonical non-zero morphism $\ic{b}\to\zz{+}{a}{b}$ is defined by the commutative diagram
                    \begin{equation*}
                        \begin{tikzcd}[column sep=large]
                            \ic{b}\arrow[dashed]{d}[swap]{\exists!}\arrow{rd}{\phi^0_{b,b+1}}&\\
                            \zz{+}{a}{b}\arrow{r}{\pi_{b+2,b}\dots\pi_{a,b}}
                            &\stan{b+1}
                        \end{tikzcd}
                    \end{equation*}
                    By induction on $a-b$, the octahedral axiom for the composition $\ic{b}\to\zz{+}{a}{b}\longto{\pi_{a,b}}\zz{+}{a-2}{b}$ yields triangles
                    \begin{equation}\label{2SESs1}
                        \ic{b}\to\zz{+}{a}{b}\to\zz{+}{a}{b+1}\to\ic{b}[1].
                    \end{equation}
                    Dually, there are also triangles $\zz{-}{a}{b+1}\to\zz{-}{a}{b}\to\ic{b}\to\zz{-}{a}{b+1}[1]$.
                    These triangles and the ones from \cref{2SESsrem} are used in \cite{CiprianiLanini} to inductively construct the string objects, starting from $\zz{\pm}{a}{a}=\ic{a}$.
                \end{remark}
            \subsubsection{Morphisms between standard objects and string objects}\label{stdzigzaghoms}
                \begin{lemma}
                    \label{homdeltazigzag}
                    For $0\leq i<\frac{a-b}{2}$ we have
                    \begin{equation*}
                        \Hom_{\proj^n}(\stan{a-2i},\zz{+}{a}{b}[r])\cong\Hom_{\proj^n}(\zz{-}{a}{b},\costan{a-2i}[r])\cong\begin{cases}
                            \kk&\text{if }r=2i,\\
                            0&\text{else}.
                        \end{cases}
                    \end{equation*}
                \end{lemma}
                \begin{proof}
                    By Verdier duality it suffices to compute $\Hom_{\proj^n}(\stan{a-2i},\zz{+}{a}{b}[r])$.
                    
                    If $a-b$ is odd, we apply $\Hom_{\proj^n}(\stan{a-2i},-)$ to the triangle \cref{2SESs1} and the first triangle in \cref{2SESs} defining $\zz{+}{a}{b+1}$.
                    Since $a-2i>b$, we have $\Hom_{\proj^n}(\stan{a-2i},\ic{b}[r])=0$ for all $r$ by \cref{topIC_to_lower_nabla}.
                    Moreover, since the $\costan{}$-flag of $\zz{-}{a-1}{b+1}$ does not involve $\nabla_{a-2i}$ it follows that $\Hom_{\proj^n}(\stan{a-2i},\zz{-}{a-1}{b+1}[r])=0$ for all $r$.
                    Together these observations imply
                    \begin{equation*}
                        \Hom_{\proj^n}(\stan{a-2i},\zz{+}{a}{b}[r])\cong\Hom_{\proj^n}(\stan{a-2i},\zz{+}{a}{b+1}[r])\cong\Hom_{\proj^n}(\stan{a-2i},\ic{a}[r]),
                    \end{equation*}
                    which has the claimed form by \cref{topIC_to_lower_nabla}.

                    The argument for $a-b$ even is similar, using the first triangle in \cref{2SESs} and the fact that $\zz{-}{a-1}{b}$ has a $\costan{}$-flag not involving $\costan{a-2i}$.
                \end{proof}
                The following lemma determines the morphisms from string modules to some standard modules.
                \begin{lemma}\label{homzigzagdelta}
                    Let $0\leq b \leq a\leq n$ and $r\in\Z$.
                    \begin{enumerate}
                        \item If $a-b$ is even, then
                            \begin{equation*}
                                \Hom_{\proj^n}(\zz{+}{a}{b},\stan{a}[r])\cong\Hom_{\proj^n}(\costan{a},\zz{-}{a}{b}[r])\cong\begin{cases}
                                    \kk&\text{if $r=a+b$},\\
                                    0&\text{otherwise}.
                                \end{cases}
                            \end{equation*}
                        \item If $a-b$ is odd, then
                            \begin{equation*}
                                \Hom_{\proj^n}(\zz{+}{a}{b},\stan{a}[r])\cong\Hom_{\proj^n}(\costan{a},\zz{-}{a}{b}[r])\cong\begin{cases}
                                    \kk&\text{if $r=a-b-1$},\\
                                    0&\text{otherwise}.
                                \end{cases}
                            \end{equation*}
                    \end{enumerate}
                \end{lemma}
                \begin{proof}
                    By Verdier duality it suffices to compute $\Hom_{\proj^n}(\zz{+}{a}{b},\stan{a}[r])$.
                    For $a=b$ and $a=b+1$, the claim follows from \cref{mor_simple_stand} and \cref{DeltaHoms}, respectively.
                    
                    For $a>b+1$, by applying the functor $\Hom_{\proj^n}(-,\stan{a})$ to the triangle $\stan{a}\to\zz{+}{a}{b}\to\zz{+}{a-2}{b}\to\stan{a}[1]$ we get the long exact sequence
                    \[
                        \ldots \to \Hom_{\proj^n}(\zz{+}{a-2}{b},\stan{a}[r]) \to \Hom_{\proj^n}(\zz{+}{a}{b},\stan{a}[r]) \to \Hom_{\proj^n}(\stan{a},\stan{a}[r]) \to \ldots.
                    \]
                    By \cref{DeltaHoms} the right-hand side is one-dimensional and concentrated in degree $0$, and hence $\Hom_{\proj^n}(\zz{+}{a}{b},\stan{a}[r])\cong\Hom_{\proj^n}(\zz{+}{a-2}{b},\stan{a}[r])$ for $r\geq 2$.
                    For $r\leq 1$ we get $\Hom_{\proj^n}(\zz{+}{a}{b},\stan{a}[r])=0$: as a consequence of \cref{stringwelldef}, the connecting morphism $\Hom_{\proj^n}(\stan{a},\stan{a})\to\Hom_{\proj^n}(\zz{+}{a-2}{b},\stan{a}[1])$ is an isomorphism, and $\Hom_{\proj^n}(\zz{+}{a-2}{b},\stan{a})=0$ by the inductive construction of $\zz{+}{a-2}{b}$ and \cref{DeltaHoms} (and \cref{mor_simple_stand} if $a-b$ is even).
                    
                    To compute $\Hom_{\proj^n}(\zz{+}{a-2}{b},\stan{a}[r])$ for $r\geq 2$, we analyze the long exact sequence obtained by applying $\Hom_{\proj^n}(\zz{+}{a-2}{b},-)$ to the triangle \cref{delta_comp_series} definying $\stan{a}$.
                    By restriction to $\proj^{a-2}$ and naturality of the adjunction, we have a commutative square
                    \begin{equation*}
                        \begin{tikzcd}
                            \Hom_{\proj^n}(\zz{+}{a-2}{b},\ic{a}[r])\arrow{r}\arrow{d}[swap]{\cong}
                            &\Hom_{\proj^n}(\zz{+}{a-2}{b},\ic{a-1}[r+1])\arrow{d}{\cong}\\
                            \Hom_{\proj^n}(\zz{+}{a-2}{b},\ic{a-2}[r-2])\arrow{r}
                            &\Hom_{\proj^n}(\zz{+}{a-2}{b},\ic{a-2}[r])
                        \end{tikzcd}
                    \end{equation*}
                    where the bottom morphism is given by postcomposition with $\imath_{a-2}^!(\epsilon^1_{a,a-1})$, which by \cref{ICdegree2end} is $\epsilon^2_{a-2,a-2}\colon \ic{a-2}[-2]\to\ic{a-2}$.
                    By \cref{HomZZtopICmodule,HomZZtopIC}, this is an isomorphism unless $r=a+b$ if $a-b$ is even, respectively unless $r=a-b-1$ if $a-b$ is odd, and the claim follows from this.
                \end{proof}
            \subsubsection{Morphisms between string objects}\label{seczigzaghoms}
                Now we can determine the morphisms between string objects.
                Although the proof is mostly the same, we need to distinguish two cases depending on the parity of $a-b$. 
                \begin{proposition}\label{zigzaghoms}
                    Let $0\leq b\leq a\leq n$, $0\leq i\leq \frac{a-b}{2}$ and $r\in\Z$.
                    \begin{enumerate}
                        \item If $a-b$ is even, then
                            \begin{equation*}
                                \Hom_{\proj^n}(\zz{+}{a-2i}{b},\zz{+}{a}{b}[r])\cong\Hom_{\proj^n}(\zz{-}{a}{b},\zz{-}{a-2i}{b}[r])\cong\begin{cases}
                                    \kk&\text{if $2i\leq r\leq a+b$ and $r$ even},\\
                                    0&\text{otherwise}.
                                \end{cases}
                            \end{equation*}
                        \item If $a-b$ is odd, then
                            \begin{equation*}
                                \Hom_{\proj^n}(\zz{+}{a-2i}{b},\zz{+}{a}{b}[r])\cong\Hom_{\proj^n}(\zz{-}{a}{b},\zz{-}{a-2i}{b}[r])\cong\begin{cases}
                                    \kk&\text{if $2i\leq r\leq a-b-1$ and $r$ even},\\
                                    0&\text{otherwise}.
                                \end{cases}
                            \end{equation*}
                    \end{enumerate}
                \end{proposition}
                \begin{proof}\leavevmode
                    \begin{enumerate}
                        \item By Verdier duality it suffices to compute $\Hom_{\proj^n}(\zz{+}{a-2i}{b},\zz{+}{a}{b}[r])$.
                            We do this by downward induction on $i$.
                            For the base case $i=\frac{a-b}{2}$ we have $\zz{+}{a-2i}{b}=\ic{b}$, so the claim follows from \cref{HomZZminusIC}.
                            
                            For the inductive step, we apply $\Hom_{\proj^n}(-,\zz{+}{a}{b})$ to the triangle \cref{delta flag of zigzags} defining $\zz{+}{a-2i}{b}$, and $\Hom_{\proj^n}(\stan{a-2i},-)$ to the triangle \cref{2SESs} to obtain the diagram
                            \begin{equation}\label{diagram1}
                                \adjustbox{scale=0.7,center}{\begin{tikzcd}
                                &&&\ldots\ar[d]&& \\
                                &&&\Hom_{\proj^n}(\stan{a-2i},\zz{-}{a-1}{b}[r])\ar[d]&&\\
                                \ldots \ar[r]&
                                \Hom_{\proj^n}(\zz{+}{a-2i-2}{b},\zz{+}{a}{b}[r])\ar[r] & \Hom_{\proj^n}(\zz{+}{a-2i}{b},\zz{+}{a}{b}[r])\ar[r] & \Hom_{\proj^n}(\stan{a-2i},\zz{+}{a}{b}[r])\ar[r]\ar[d] &\ldots\\
                                &&&\Hom_{\proj^n}(\stan{a-2i},\ic{a}[r])\ar[d]&&\\
                                &&&\ldots&&
                                \end{tikzcd}}
                            \end{equation}
                            By the inductive hypothesis $\Hom_{\proj^n}(\zz{+}{a-2i-2}{b},\zz{+}{a}{b}[r])=\kk$ for $r\in\{2i+2,2i+4,\dots,a+b\}$, and it vanishes otherwise.
                            Since $\zz{-}{a-1}{b}$ has a $\costan{}$-flag not involving $\costan{a-2i}$ by the dual version of \cref{delta flag of zigzags}, the first term in the column of \cref{diagram1} vanishes (this follows either by direct calculation, or from e.g.~\cite[Thm.~3.11]{BrundanStroppelHW}).
                            Moreover, since by \cref{topIC_to_lower_nabla} the last term in the column of \cref{diagram1} is one dimensional and concentrated in degree $2i$, so is $\Hom_{\proj^n}(\stan{a-2i},\zz{+}{a}{b}[r])$, from which the claim follows.
                        \item By Verdier duality it suffices to compute $\Hom_{\proj^n}(\zz{+}{a-2i}{b},\zz{+}{a}{b}[r])$.
                            We do this by downward induction on $i$.
                            In the base case $i=\frac{a-b-1}{2}$ we have $\zz{+}{a-2i}{b}=\stan{b+1}$, so this follows from \cref{homdeltazigzag}.
    
                            For the inductive step, we apply the functor $\Hom_{\proj^n}(-,\zz{+}{a}{b})$ to the triangle \cref{delta flag of zigzags}. 
                            By induction, $\Hom_{\proj^n}(\zz{+}{a-2i-2}{b},\zz{+}{a}{b}[r])\cong\kk$ for $2i+2\leq r\leq a-b-1$ and $r$ even.
                            By \cref{homdeltazigzag} we know that $\Hom_{\proj^n}(\stan{a-2i},\zz{+}{a}{b}[r])\cong\kk$ for $r=2i$, and that it vanishes otherwise, which implies the claim.\qedhere
                    \end{enumerate}
                \end{proof}
                \begin{remark}
                    \label{zigzaghomdiags}
                    From the proof of \cref{zigzaghoms} we obtain the following descriptions of canonical non-zero morphisms $\Phi^{2i}_{a,b}\colon \zz{+}{a}{b}\to\zz{+}{a}{b}[2i]$, $\ol{\Phi}^{2i}_{a-2i,b}\colon \zz{+}{a-2i}{b}\to\zz{+}{a}{b}[2i]$ and $\zeta^{2i}_{a-2i,b}\colon\stan{a-2i}\to\zz{+}{a}{b}[2i]$:
                    for $0\leq i\leq\frac{a-b}{2}$, they arise from the diagram
                    \begin{equation*}
                        \begin{tikzcd}
                            \zz{+}{a}{b}\arrow[dashed]{rr}{\exists!\Phi^{2i}_{a,b}}\arrow{d}[swap]{\pi_{a-2i+2,b}\dots\pi_{a,b}}
                            &&\zz{+}{a}{b}[2i]\arrow{d}{n^0_{a,b}}\\
                            \zz{+}{a-2i}{b}\arrow[dashed]{rru}{\exists!\ol{\Phi}^{2i}_{a-2i,b}}
                            &&\ic{a}[2i]\\
                            &\stan{a-2i}\arrow{lu}{\iota_{a-2i,b}}\arrow{ru}[swap]{\mu_{a-2i,a}}\arrow[dashed]{ruu}[near start]{\exists!\zeta^{2i}_{a-2i,b}}&
                        \end{tikzcd}
                    \end{equation*}
                    Note that this diagram does not depend on whether $a-b$ is even or odd, though in the proof of \cref{zigzaghoms} one has to distinguish this since the induction results in different base cases.

                    Moreover, by \cref{HomZZtopICdiag} $\Hom_{\proj^n}(\zz{+}{a-2i}{b},\ic{a}[2i])$ is spanned by the composition $\zz{+}{a-2i}{b}\longto{n^0_{a-2i,b}}\ic{a-2i}\longto{\epsilon^{2i}_{a-2i,a}}\ic{a}[2i]$.
                    Since we have $n^0_{a-2i,b}\iota_{a-2i,b}=\mu_{a-2i,a-2i}$ by construction, it follows that $\epsilon^{2i}_{a-2i,a}n^0_{a-2i,b}\iota_{a-2i,b}=\mu_{a-2i,a}$.
                    An easy diagram chase (using that $\Hom_{\proj^n}(\zz{+}{a-2i}{b},\ic{a}[2i])\cong\kk$) then shows that $\epsilon^{2i}_{a-2i,a}n^0_{a-2i,b}\colon \zz{+}{a-2i}{b}\to\ic{a}[2i]$ makes the entire diagram above commute, and therefore the morphism $\Phi^{2i}_{a,b}\colon \zz{+}{a}{b}\to\zz{+}{a}{b}[2i]$ can also be defined by the diagram
                    \begin{equation*}
                        \begin{tikzcd}[column sep=width("aaaaaaaaaaaa")]
                            \zz{+}{a}{b}\arrow[dashed]{r}{\exists!\Phi^{2i}_{a,b}}\arrow{d}[swap]{\pi_{a-2i+2,b}\dots\pi_{a,b}}
                            &\zz{+}{a}{b}[2i]\arrow{d}{n^0_{a,b}}\\
                            \zz{+}{a-2i}{b}\arrow{r}{\epsilon^{2i}_{a-2i,a}n^0_{a-2i,b}}
                            &\ic{a}[2i]\rlap{.}
                        \end{tikzcd}
                    \end{equation*}
                    For $i>\frac{a-b}{2}$ (this can only happen if $a-b$ is even), the morphism $\Phi^{2i}_{a,b}$ arises from the diagram
                    \begin{equation*}
                        \begin{tikzcd}[column sep=width("aaaaaaa")]
                            \zz{+}{a}{b}\arrow[dashed]{r}{\exists!\Phi^{2i}_{a,b}}\arrow{d}[swap]{\pi_{b+2,b}\dots\pi_{a,b}}
                            &\zz{+}{a}{b}[2i]\arrow{d}{n^0_{a,b}}\\
                            \ic{b}\arrow{r}{\epsilon^{2i}_{b,a}}\arrow[dashed]{ru}{\exists!}
                            &\ic{a}[2i]\rlap{.}
                        \end{tikzcd}
                    \end{equation*}
                    The morphisms $\zz{-}{a}{b}\to\zz{-}{a}{b}[2i]$ are described by the dual diagrams.
                \end{remark}
            \subsubsection{Determining the composition}\label{zigzagcomposition}
                To show that the string objects are $\proj$-like, we are left to determine the composition of morphisms in $\End_{\proj^n}^*(\zz{\pm}{a}{b})$.
                This requires the following two technical lemmas, which show that the morphisms from \cref{zigzaghomdiags} are compatible with the quotient maps between the string objects and the morphisms betwen IC sheaves.
                \begin{lemma}
                    \label{criticallemma}
                    For $0\leq b\leq a\leq n$ and $1\leq i\leq\frac{a-b-2}{2}$ if $a-b$ is even, respectively $1\leq i\leq\frac{a-b-3}{2}$ if $a-b$ is odd, the square
                    \begin{equation*}
                        \begin{tikzcd}[column sep=large]
                            \zz{+}{a-2}{b}\arrow{r}{\ol{\Phi}^2_{a-2,b}}\arrow{d}[swap]{\pi_{a-2i,b}\dots\pi_{a-2,b}}
                            &\zz{+}{a}{b}[2]\arrow{d}{\pi_{a-2i+2,b}\dots\pi_{a,b}}\\
                            \zz{+}{a-2i-2}{b}\arrow{r}{\ol{\Phi}^2_{a-2i-2,b}}
                            &\zz{+}{a-2i}{b}[2]
                        \end{tikzcd}
                    \end{equation*}
                    commutes up to a non-zero scalar.
                \end{lemma}
                \begin{proof}
                    We prove the claim by induction on $i$.
                    For $i=1$, the composition $\Phi^2_{a-2,b}=\ol{\Phi}^2_{a-4,b}\pi_{a-2,b}$ spans $\Hom_{\proj^n}(\zz{+}{a-2}{b},\zz{+}{a-2}{b}[2])$ by \cref{zigzaghomdiags,zigzaghoms}.
    
                    We want to show that this morphism factors through $\zz{+}{a}{b}[2]$.
                    Completing the right column of the square to the triangle $\stan{a}[2]\longto{\iota_{a,b}}\zz{+}{a}{b}[2]\longto{\pi_{a,b}}\zz{+}{a-2}{b}[2]\longto{\psi_{a-2,b}}\stan{a}[3]$ and applying $\Hom_{\proj^n}(\zz{+}{a-2}{b},-)$ yields the exact sequence
                    \begin{equation*}
                        \Hom_{\proj^n}(\zz{+}{a-2}{b},\zz{+}{a}{b}[2])\to \Hom_{\proj^n}(\zz{+}{a-2}{b},\zz{+}{a-2}{b}[2])\to \Hom_{\proj^n}(\zz{+}{a-2}{b},\stan{a}[3]).
                    \end{equation*}
                    By \cref{zigzaghoms} the first two terms are $1$-dimensional, and by the same argument as in the proof of \cref{homzigzagdelta} we get $\Hom_{\proj^n}(\zz{+}{a-2}{b},\stan{a}[3])=0$, which yields the required factorization.
                    
                    For the inductive step, we use the diagram
                    \begin{equation*}
                        \begin{tikzcd}[column sep=width("aaaaaaaaaaaaa")]
                            \zz{+}{a-2}{b}\arrow{r}{\pi_{a-2i,b}\dots\pi_{a-2,b}}\arrow{d}[swap]{\ol{\Phi}^2_{a-2,b}}
                            &\zz{+}{a-2i-2}{b}\arrow{r}{\pi_{a-2i-2,b}}\arrow{d}{\ol{\Phi}^2_{a-2i-2,b}}
                            &\zz{+}{a-2i-4}{b}\arrow{d}{\ol{\Phi}^2_{a-2i-4,b}}\\
                            \zz{+}{a}{b}[2]\arrow{r}{\pi_{a-2i+2,b}\dots\pi_{a,b}}
                            &\zz{+}{a-2i}{b}[2]\arrow{r}{\pi_{a-2i,b}}
                            &\zz{+}{a-2i-2}{b}[2]\rlap{.}
                        \end{tikzcd}
                    \end{equation*}
                    By the base case, the right square commutes, and by induction the left square commutes (in both cases up to a non-zero scalar).
                    Thus the outer rectangle commutes up to a non-zero scalar by an easy diagram chase.
                \end{proof}
                \begin{lemma}\label{secondcriticallemma}
                   For $0<b<a$ and $a-b$ even, the square
                    \begin{equation*}
                        \begin{tikzcd}
                            \zz{+}{a}{b}\arrow{d}[swap]{\pi_{b+2,b}\dots\pi_{a,b}}\arrow{r}{\Phi^2_{a,b}}
                            &\zz{+}{a}{b}[2]\arrow{d}{\pi_{b+2,b}\dots\pi_{a,b}}\\
                            \ic{b}\arrow{r}{\epsilon^{2}_{b,b}}
                            &\ic{b}[2]
                        \end{tikzcd}
                    \end{equation*}
                    commutes up to a non-zero scalar. 
                \end{lemma}
                \begin{proof}
                    The bottom left path is non-zero by \cref{HomZZplusICmodule}, and hence we have to show that it factors through the vertical morphism on the right.
                    By the triangle $\zz{+}{a}{b+1}\to\zz{+}{a}{b}\to\ic{b}\to\zz{+}{a}{b+1}[1]$, it suffices to show $\Hom(\zz{+}{a}{b},\zz{+}{a}{b+1}[3])=0$.
                    We prove this by induction on $a-b$.
                    In the base case $a=b+2$ we have $\zz{+}{a}{b+1}=\stan{a}$, and thus the claim follows from \cref{homzigzagdelta}.
                    
                    For the inductive step, assume that $a>b+2$.
                    We apply the functor $\Hom(\zz{+}{a}{b},-)$ to the triangle $\stan{a}\to \zz{+}{a}{b+1}\to\zz{+}{a-2}{b+1}\to\stan{a}[1]$ defining $\zz{+}{a}{b+1}$ to get the long exact sequence
                    \begin{equation*}
                        \ldots\to\Hom(\zz{+}{a}{b},\stan{a}[r])\to\Hom(\zz{+}{a}{b},\zz{+}{a}{b+1}[r])\to\Hom(\zz{+}{a}{b},\zz{+}{a-2}{b+1}[r])\to\ldots.
                    \end{equation*}
                    By \cref{homzigzagdelta} we have $\Hom(\zz{+}{a}{b},\stan{a}[r])=0$ unless $r=a+b$, and therefore we have $\Hom(\zz{+}{a}{b},\zz{+}{a}{b+1}[3])\cong\Hom(\zz{+}{a}{b},\zz{+}{a-2}{b+1}[3])$ (note that $a>b+2$ and $a-b$ even implies $a+b>4$).
                    
                    To compute $\Hom(\zz{+}{a}{b},\zz{+}{a-2}{b+1}[3])$, we use the long exact sequence obtained by applying $\Hom(-,\zz{+}{a-2}{b+1})$ to the triangle $\stan{a}\to \zz{+}{a}{b}\to\zz{+}{a-2}{b}\to\stan{a}[1]$.
                    By \cref{DeltaHoms} and the inductive construction of $\zz{+}{a-2}{b+1}$, we have $\Hom(\stan{a},\zz{+}{a-2}{b+1}[r])=0$ for all $r$, and therefore in particular
                    \begin{equation*}
                        \Hom(\zz{+}{a}{b},\zz{+}{a-2}{b+1}[3])\cong\Hom(\zz{+}{a-2}{b},\zz{+}{a-2}{b+1}[3])
                    \end{equation*}
                    which vanishes by induction on $a-b$.
                \end{proof}
                \begin{theorem}\label{stringsplike}
                    Let $0\leq b\leq a\leq n$.
                    \begin{enumerate}
                        \item If $a-b$ is even, then $\zz{\pm}{a}{b}$ is $\proj^{(a+b)/2}$-like.
                        \item If $a-b$ is odd, then $\zz{\pm}{a}{b}$ is $\proj^{(a-b-1)/2}$-like.
                    \end{enumerate}
                \end{theorem}
                \begin{proof}
                    By Verdier duality it suffices to show that $\zz{+}{a}{b}$ is $\proj$-like.
                    By \cref{zigzaghoms} we have isomorphisms of graded $\kk$-vector spaces $\End_{\proj^n}^*(\zz{+}{a}{b})\cong \kk[t]/(t^{(a+b)/2+1})$ if $a-b$ is even (respectively $\End_{\proj^n}^*(\zz{+}{a}{b})\cong \kk[t]/(t^{(a-b-1)/2+1})$ if $a-b$ is odd) with $\deg(t)=2$.
                    Therefore we only have to show that the composition of the canonical morphisms $\zz{+}{a}{b}\longto{\Phi^2_{a,b}}\zz{+}{a}{b}[2]\longto{\Phi^{2i}_{a,b}}\zz{\pm}{a}{b}[2i+2]$ is non-zero for $0<i<\frac{a+b}{2}$ if $a-b$ is even (respectively $0<i<\frac{a-b-1}{2}$ if $a-b$ is odd).
                    We show that up to a non-zero scalar this composition agrees with the canonical morphism $\Phi^{2i+2}_{a,b}\colon \zz{\pm}{a}{b}\to\zz{\pm}{a}{b}[2i+2]$.

                    First assume that $i<\frac{a-b}{2}$.
                    By \cref{zigzaghomdiags}, the composition $\Phi^{2i}_{a,b}\Phi^2_{a,b}$ and the morphism $\Phi^{2i+2}_{a,b}$ (which is not drawn) are defined by the non-dashed arrows in the diagram
                    \begin{equation*}
                    \adjustbox{scale=0.6,center}{
                        \begin{tikzcd}[column sep=width("aaaaaaaaaaaa"),row sep=large,bend angle=15]
                            \zz{+}{a}{b}\arrow{rrr}{\Phi^2_{a,b}}\arrow{rd}{\pi_{a,b}}\arrow[two heads]{dd}[swap]{\pi_{a-2i,b}\dots\pi_{a,b}}
                            &&&\zz{+}{a}{b}[2]\arrow{rrr}{\Phi^{2i}_{a,b}}\arrow{rd}{\pi_{a-2i+2,b}\dots\pi_{a,b}}\arrow{ld}{n^0_{a,b}}
                            &&&\zz{+}{a}{b}[2i+2]\arrow{dd}[swap]{n^0_{a,b}}\arrow{ld}{n^0_{a,b}}\\
                            &\zz{+}{a-2}{b}\arrow{r}{\epsilon^2_{a-2,a}n^0_{a-2,b}}\arrow{ld}{\pi_{a-2i,b}\dots\pi_{a-2,b}}
                            &\ic{a}[2]
                            &&\zz{+}{a-2i}{b}[2]\arrow{r}{\epsilon^{2i}_{a-2i,a}n^0_{a-2i,b}}
                            &\ic{a}[2i+2]\arrow{rd}{=}&\\
                            \zz{+}{a-2i-2}{b}\arrow{rrrrrr}{\epsilon^{2i+2}_{a-2i-2,a}n^0_{a-2i-2,b}}\arrow[dashed]{rrrru}{\ol{\Phi}^2_{a-2i-2,b}}
                            &&&&&&\ic{a}[2i+2]
                        \end{tikzcd}}
                    \end{equation*}
                    We moreover have the canonical morphism $\ol{\Phi}^2_{a-2i-2,b}\colon \zz{+}{a-2i-2}{b}\to\zz{+}{a-2i}{b}[2]$ from \cref{zigzaghomdiags} (the dashed arrow in the diagram).
                    \begin{steps}
                        \item The square
                            \begin{equation*}
                                \begin{tikzcd}[column sep=width("aaaaaaaaa")]
                                    \zz{+}{a}{b}\arrow{r}{\Phi^2_{a,b}}\arrow{d}[swap]{\pi_{a-2i,b}\dots\pi_{a,b}}
                                    &\zz{+}{a}{b}[2]\arrow{d}{\pi_{a-2i+2,b}\dots\pi_{a,b}}\\
                                    \zz{+}{a-2i-2}{b}\arrow[dashed]{r}{\ol{\Phi}^2_{a-2i-2,b}}
                                    &\zz{+}{a-2i}{b}[2]
                                \end{tikzcd}
                            \end{equation*}
                            commutes up to a non-zero scalar.

                            Proof: From the construction of $\Phi^2_{a,b}$ in \cref{zigzaghomdiags} we obtain the diagram
                            \begin{equation*}
                                \begin{tikzcd}[column sep=width("aaaaaaaaa")]
                                    \zz{+}{a}{b}\arrow[bend left=10]{rrd}{\Phi^2_{a,b}}\arrow[bend right]{rdd}[swap]{\pi_{a-2i,b}\dots\pi_{a,b}}\arrow{rd}{\pi_{a,b}}&&\\
                                    &\zz{+}{a-2}{b}\arrow{d}[swap]{\pi_{a-2i,b}\dots\pi_{a-2,b}}\arrow{r}{\ol{\Phi}^2_{a-2,b}}
                                    &\zz{+}{a}{b}[2]\arrow{d}{\pi_{a-2i+2,b}\dots\pi_{a,b}}\\
                                    &\zz{+}{a-2i-2}{b}\arrow[dashed]{r}{\ol{\Phi}^2_{a-2i-2,b}}
                                    &\zz{+}{a-2i}{b}[2]\rlap{,}
                                \end{tikzcd}
                            \end{equation*}
                            where the triangle at the top commutes by definition of $\Phi^2_{a,b}$ and the triangle at the left commutes obviously.
                            By \cref{criticallemma} the inner square commutes up to a non-zero scalar, and the claim then follows by an easy diagram chase.
                        \item The diagram
                            \begin{equation*}
                                \begin{tikzcd}[column sep=huge]
                                    \zz{+}{a-2i-2}{b}\arrow[dashed]{r}{\ol{\Phi}^2_{a-2i-2,b}}\arrow{rd}[swap]{\epsilon^{2i+2}_{a-2i-2,a}n^0_{a-2i-2,b}}
                                    &\zz{+}{a-2i}{b}[2]\arrow{d}{\epsilon^{2i}_{a-2i,a}n^0_{a-2i,b}}\\
                                    &\ic{a}[2i+2]
                                \end{tikzcd}
                            \end{equation*}
                            commutes.

                            Proof: Since $\epsilon^{2i+2}_{a-2i-2,a}=\epsilon^{2i}_{a-2i,a}\epsilon^2_{a-2i-2,a-2i}$, it suffices to show that the diagram
                            \begin{equation*}
                                \begin{tikzcd}[column sep=huge]
                                    \zz{+}{a-2i-2}{b}\arrow[dashed]{r}{\ol{\Phi}^2_{a-2i-2,b}}\arrow{d}[swap]{n^0_{a-2i-2,b}}
                                    &\zz{+}{a-2i}{b}[2]\arrow{d}{n^0_{a-2i,b}}\\
                                    \ic{a-2i-2}\arrow{r}{\epsilon^2_{a-2i-2,a-2i}}
                                    &\ic{a-2i}[2]
                                \end{tikzcd}
                            \end{equation*}
                            commutes.
                            For this, the construction of the morphism $\ol{\Phi}^2_{a-2i-2,b}\colon \zz{+}{a-2i-2}{b}\to\zz{+}{a-2i}{b}[2]$ from \cref{zigzaghomdiags} gives the diagram
                            \begin{equation*}
                                \begin{tikzcd}[column sep=huge]
                                    \stan{a-2i-2}\arrow{rd}{\iota_{a-2i-2,b}}\arrow[bend left=10]{rrd}{\zeta^{2i}_{a-2i,b}}\arrow[bend right]{rdd}[swap]{\mu_{a-2i-2,a-2i-2}}&&\\
                                    &\zz{+}{a-2i-2}{b}\arrow{r}{\ol{\Phi}^2_{a-2i-2,b}}\arrow{d}[swap]{n^0_{a-2i-2,b}}
                                    &\zz{+}{a-2i}{b}[2]\arrow{d}{n^0_{a-2i,b}}\\
                                    &\ic{a-2i-2}\arrow{r}{\epsilon^2_{a-2i-2,a-2i}}
                                    &\ic{a-2i}[2]
                                \end{tikzcd}
                            \end{equation*}
                            where the outer square and the triangle at the top and the left commute by construction.
                            
                            By \cref{HomZZtopICdiag}, we know that $\Hom(\zz{+}{a-2i-2}{b},\ic{a-2i}[2])\cong\kk$ is spanned by  the composition $\epsilon^2_{a-2i-2,a-2i}n^0_{a-2i-2,b}\colon\zz{+}{a-2i-2}{b}\to\ic{a-2i}[2]$.
                            Moreover, this morphism is uniquely characterized by its precomposition with $\iota_{a-2i-2,b}\colon \stan{a-2i-2}\hookto\zz{+}{a-2i-2}{b}$, which by the proof of \cref{HomZZtopIC} gives $\mu_{a-2i-2,a-2i}\colon \stan{a-2i-2}\to\ic{a-2i}[2]$.
                            From the diagram it follows that
                            \begin{equation*}
                                n^0_{a-2i,b}\ol{\Phi}^2_{a-2i-2,b}\iota_{a-2i-2,b}=\epsilon^2_{a-2i-2,a-2i}\mu_{a-2i-2,a-2i-2}=\mu_{a-2i-2,a-2i},
                            \end{equation*}
                            and therefore $n^0_{a-2i,b}\ol{\Phi}^2_{a-2i-2,b}=\epsilon^2_{a-2i-2,a-2i}n^0_{a-2i-2,b}$, as required.
                        \item The claim then follows from Steps~1 and 2 by a straightforward diagram chase, using the definition of $\Phi^{2i+2}_{a,b}$ from \cref{zigzaghomdiags}.
                    \end{steps}
                    For $i\geq \frac{a-b}{2}$ the proof is very similar.
                    First, if $a=b$, then $\zz{+}{a}{b}=\ic{b}$ and $\Phi^2_{b,b}=\epsilon^2_{b,b}$, so there is nothing to show.
                    For $a>b$, the ``big diagram'' is almost the same, except that $\zz{+}{a-2i}{b}$ and $\zz{+}{a-2i-2}{b}[2]$ have to be replaced by $\ic{b}$ and $\ic{b}[2]$, respectively, and one has to use $\epsilon^2_{b,b}\colon\ic{b}\to\ic{b}[2]$ as the ``dashed morphism''.
                    This satisfies (by construction) $\epsilon^{2i+2}_{b,a}=\epsilon^{2i}_{b,a}\epsilon^2_{b,b}$.
                    By \cref{secondcriticallemma} the square 
                    \begin{equation*}
                        \begin{tikzcd}
                            \zz{+}{a}{b}\arrow{d}[swap]{\pi_{b+2,b}\dots\pi_{a,b}}\arrow{r}{\Phi^{2i}_{a,b}}
                            &\zz{+}{a}{b}[2]\arrow{d}{\pi_{b+2,b}\dots\pi_{a,b}}\\
                            \ic{b}\arrow{r}{\epsilon^{2}_{b,b}}
                            &\ic{b}[2]
                        \end{tikzcd}
                    \end{equation*}
                    commutes up to a non-zero scalar (note that if $\frac{a-b}{2}\leq i<\frac{a+b}{2}$, then $b>0$), and the claim follows from this by an easy diagram chase.\qedhere
                \end{proof}
                \begin{remark}
                    It would be desirable to show that the squares in \cref{criticallemma,secondcriticallemma} commute (not only up to non-zero scalar).
                    This cannot be obtained from our argument, which in both cases uses that morphisms to some cone must vanish to get a factorization.
                    If these squares actually commute, then the proof of \cref{stringsplike} shows that the canonical morphisms $\Phi^{2i}_{a,b}\colon\zz{+}{a}{b}\to\zz{+}{a}{b}[2i]$ form a multiplicative basis of $\End_{\proj^n}^*(\zz{+}{a}{b})$.
                \end{remark}
                Recall that $\proj^1$-objects and $\proj^1$-like objects are also known as spherical and spherelike objects, and that $\proj^0$-like objects are also known as exceptional objects.
                From \cref{stringsplike,CY_objects} one can easily read off the spherelike, spherical and exceptional string objects.
                Note that by \cref{zigzaghoms} spherelike string objects in $\perv{\proj^n}$ are necessarily $2$-spherelike.
                \begin{corollary}\label{sphericalexceptionals}
                    For the string objects in $\perv{\proj^n}$ we have:
                    \begin{enumerate}
                        \item $\zz{\pm}{a}{b}$ is $2$-spherelike if and only if $a-b=3$, or $a=2$ and $b=0$, or $a=b=1$.
                        \item The only $2$-spherical string object is $\zz{\pm}{1}{1}=\ic{1}$ for $n=1$.
                        \item The only string objects that are exceptional are the standard objects and the costandard objects.
                    \end{enumerate}
                \end{corollary}
                Combining this with the classification of indecomposable perverse sheaves mentioned in \cref{otherperspectives} (which is obtained from the description in terms of finite-dimensional algebras), this in particular recovers the classification of exceptional objects from \cite[Prop.~3]{MR4090927}.
                Moreover, we also obtain:
                \begin{corollary}\label{onlyPlikes}
                    All indecomposable objects in $\perv{\proj^n}$ are either $\proj$-like or $0$-spherical.
                \end{corollary}
                \begin{proof}
                    By the classification of indecomposable objects over a special biserial algebra from \cite[p.~161, Thm.]{ButlerRingel} and \cite[Prop.~2.3]{WaldWaschbuesch}, the indecomposable objects in $\perv{\proj^n}$ are the string objects $\zz{\pm}{a}{b}$ for $0\leq a\leq b\leq n$ and the indecomposable projective-injective objects $P_k=I_k$ for $0\leq k<n$.
                    The string objects are $\proj$-like by \cref{stringsplike}, and that $P_k$ is $0$-spherical is obvious from the description in \cref{SimpleStdProj} and \cref{projinjCY}.
                \end{proof}
    \printbibliography
\end{document}